\documentclass[reqno, 12pt]{amsart}

\usepackage{amsmath, amssymb, amsthm, amsfonts, txfonts, pxfonts}
\usepackage[english]{babel}
\usepackage{fullpage}
\usepackage{enumerate}

\usepackage[centertags]{amsmath}
\usepackage{indentfirst}
\usepackage{fancyhdr}
\usepackage[dvips]{graphicx}
\usepackage{psfrag}
\numberwithin{equation}{section}
\usepackage[latin1]{inputenc}

\usepackage[all]{xy}
\usepackage{hyperref}

\title{A $4$-sphere with non central radius and its instanton sheaf}
\author{Lucio S. Cirio, Chiara Pagani}

\address[]{\textit{Lucio Simone Cirio} \newline \indent 
Mathematisches Institut,  Georg-August Universit\"at G\"ottingen,
\newline \indent
Bunsenstra\ss e~3-5,~37073~G\"ottingen,~Germany. }
\email{lcirio@uni-math.gwdg.de }

\address[]{\textit{Chiara Pagani} \newline \indent   Universit\'e de Louvain, Institut de recherche en math\'ematique et physique,
\newline \indent
Chemin du Cyclotron 2 bte L7.01.02, 1348 Louvain-la-Neuve, Belgium. \vspace{3pt}
\newline \indent
Mathematisches Institut,  Georg-August Universit\"at G\"ottingen,
\newline \indent
Bunsenstra\ss e~3-5,~37073~G\"ottingen,~Germany. 
}
\email{cpagani@uni-math.gwdg.de}

\theoremstyle{plain}
\newtheorem{thm}{Theorem}[section]
\newtheorem{lem}[thm]{Lemma}
\newtheorem{prop}[thm]{Proposition}
\newtheorem{cor}[thm]{Corollary}
\newtheorem{defi}[thm]{Definition}

\newtheorem{ex}[thm]{Example}
\newtheorem{Q}{Question/Comment}
\newtheorem{rem}[thm]{Remark}

\renewcommand{\i}{\mathrm{i}}

\newcommand{\IR}{{\mathbb{R}}}

\newcommand{\IC}{{\mathbb{C}}}

\newcommand{\nn}{\nonumber}
\newcommand{\ot}{\otimes}
\newcommand{\pot}{\overset{.}{\otimes}}

\newcommand{\be}{\begin{equation}}
\newcommand{\ee}{\end{equation}}
\newcommand{\ra}{\rightarrow}
\newcommand{\tn}{\otimes}
\newcommand{\id}{\mathrm{id}}
\newcommand{\A}{\mathcal{A}}
\newcommand{\V}{\mathcal{V}}
\renewcommand{\Q}{\mathcal{Q}}
\renewcommand{\O}{\mathcal{O}}
\newcommand{\Mm}{\mathcal{M}}
\newcommand{\M}[1]{\mathbb{#1}}
\newcommand{\La}[1]{\Lambda_{\scriptscriptstyle{#1}}}
\newcommand{\sph}{\A(S^4_q)}
\newcommand{\rn}{~_{\scriptscriptstyle{N}}\mathbb{R}_q^4 }
\newcommand{\rs}{~_{\scriptscriptstyle{S}}\mathbb{R}_q^4 }

\newcommand{\rns}{~_{\scriptscriptstyle{NS}}\mathbb{R}_q^4 }
\newcommand{\rsn}{~_{\scriptscriptstyle{SN}}\mathbb{R}_q^4 }
\newcommand{\sqrsn}{\widetilde{~_{\scriptscriptstyle{SN}}\mathbb{R}_q^4 }}
\newcommand{\sqrns}{\widetilde{~_{\scriptscriptstyle{NS}}\mathbb{R}_q^4 }}

\renewcommand{\a}[1]{\alpha_{\scriptscriptstyle{#1}}}
\renewcommand{\b}[1]{\beta_{\scriptscriptstyle{#1}}}
\newcommand{\x}[1]{x_{\scriptscriptstyle{#1}}}
\newcommand{\y}[1]{y_{\scriptscriptstyle{#1}}}

\newcommand{\al}{\alpha}
\newcommand{\ali}{\beta}
\renewcommand{\aa}{\mathcal{A}}

\newcommand{\I}{\mathcal{I}}
\renewcommand{\P}{\mathcal{P}}
\newcommand{\co}{\mathcal{B}}
\newcommand{\can}{\chi}

\newcommand{\shs}{\O_{S^4_q}}
\newcommand{\dfs}{\Omega^{\bullet}_{S^4_q}}
\newcommand{\ndf}{_N\Omega^{\bullet}_q}
\newcommand{\sdf}{_S\Omega^{\bullet}_q}
\newcommand{\sndf}{_{SN}\Omega^{\bullet}_q}

\newcommand{\sndfu}{_{SN}\Omega^1_q}

\renewcommand{\d}{\mathrm{d}}

\newcommand{\fns}{f_{\scriptscriptstyle{NS}} }
\newcommand{\fsn}{f_{\scriptscriptstyle{SN}} }

\newcommand{\ai}{\aa \ot_\Psi \I}
\newcommand{\tnn}{\tau_{\scriptscriptstyle{NN}}}
\newcommand{\tss}{\tau_{\scriptscriptstyle{SS}}}
\newcommand{\tns}{\tau_{\scriptscriptstyle{NS}}}
\newcommand{\tsn}{\tau_{\scriptscriptstyle{SN}}}

\newcommand{\zero}[1]{{#1}_{\scriptscriptstyle{(0)}}}
\newcommand{\uno}[1]{{#1}_{\scriptscriptstyle{(1)}}}
\newcommand{\muno}[1]{{#1}_{\scriptscriptstyle{(-1)}}}

\newcommand{\con}{\mathsf{A}}
\newcommand{\ma}{\mathbf{A}}

\begin{document}

\begin{abstract}
We build an $SU(2)$-Hopf bundle over a quantum toric four-sphere whose radius is non central.
The construction is carried out using local methods in terms of sheaves of Hopf-Galois extensions. 
The associated instanton bundle is presented and endowed with a connection with anti-selfdual curvature. \\

\end{abstract}

\maketitle
\noindent

\tableofcontents

\section{Introduction}

Gauge theories are one of the most beautiful examples of the fruitful interplay between mathematics and physics. The study of instantons in terms of bundles and connections with (anti)selfdual curvatures goes back to the late 70's. The prime example is that of $SU(2)$-instantons on the four-sphere $S^4$ and their underlying Hopf bundle $S^7 \ra S^4$ \cite{at}.

In the last years various constructions of instanton bundles on noncommutative spheres  have been provided by drawing on different methods accordingly to the 'kind' of noncommutativity into play (isospectral, quantum groups, twists,...).

In this paper we deal with a noncommutative four sphere $\sph$ that was first introduced in 
\cite{clsII} in the framework of 2-cocycle deformations of toric varieties. Noncommutative toric varieties belong to a vast class of noncommutative spaces obtained by using Drinfeld's general theory of twists (or 2-cocycles) as underlying source of the deformation.  In the specific case treated there, the Hopf algebra of symmetries on which the 2-cocycle is based is the algebraic torus $(\mathbb{C}^{\times})^n$. Twisting the symmetry induces a functorial deformation in the category of its (co)modules; namely,  every (co)algebra carrying a (co)action of $(\mathbb{C}^{\times})^n$ can be endowed with a new non(co)commutative (co)product (whose explicit expression depends on the Drinfeld twist) compatible with the (co)action of the twisted symmetry. The algebra $\sph$ and its local patches we present below  (and essentially every noncommutative algebra throughout the paper) fit within this machinery. However we will not stress this formalism in the present paper, and we refer the reader to \cite{clsI,clsII} for all the details relevant to these specific deformations. 

A key feature of the sphere $\sph$ is the noncommutativity of the radius, i.e. the algebra generator 
playing the role of the radius does not belong to the center of the algebra. Related to this is also the fact that the sphere is more effectively described in terms of local charts (see \S \ref{sect:4sphere} below). A main consequence of these peculiarities is that previous approaches adopted in the study of instantons on quantum spheres (for instance the construction of a globally defined instanton projector as in \cite{cl00,lpr06}) cannot be used here. The problem is overcome by making use of sheaf theory. We look at $\sph$ as a 'locally ringed quantum space' and use the noncommutative sheaf-theoretical methods developed in \cite{pflaum} to assemble a noncommutative Hopf bundle on it. We then show that this bundle is a sheaf of Hopf-Galois extensions.
\\

The paper is organized as follows. A first section (\S \ref{sect:4sphere}) is dedicated to recall the $4$-sphere $\sph$ with its  description in terms of (isomorphic)  'local patches' $\rn,~\rs$ obtained through noncommutative localization techniques.  
In \S \ref{sec-sheaf} we give a characterization of $\sph$ (and its differential calculus) in terms of sheaves of quantum algebras on the classical sphere $S^4$. Such a sheaf-theoretic description of quantum spaces, that well befits  our $\sph$,  is based on \cite{pflaum}.   
Crucial for the construction of a noncommutative Hopf bundle is the factorization of the intersection of the two local charts  $\rn,~\rs$ as a twisted tensor product of the  algebra $\aa=\aa(SU(2))$ of coordinate functions on $SU(2)$ and a 1-dimensional interval $I$. The two algebras $\aa$ and $ I$ are both commutative, but their twisted tensor product has a noncommutative algebra  multiplication  (cf. Prop. \ref{prop:fact}). This factorization is used in  $\S \ref{sec-Hopf}$ to introduce 'transition functions' allowing for the reconstruction of a quantum principal $\aa(SU(2))$-bundle $\mathcal{P}$ on $\sph$ out of trivial bundles on the local charts. The (noncommutative) algebraic language for principal bundles is that of Hopf-Galois extensions. We prove that the sheaf of the quantum principal $\aa(SU(2)$-bundle is a sheaf of Hopf-Galois extensions.
We conclude the section by constructing a sheaf of algebras associated to the fundamental (co)representation of $SU(2)$ on $\IC^2$, playing the role of the basic instanton bundle on $\sph$, and by sketching how to get instanton bundles with higher instanton numbers. 
In the last section, \S \ref{sec:conn}, we define an $su(2)$-valued connection one form and prove that its curvature is anti-selfdual with respect to the Hodge $*$-operator of the canonical differential calculus on $\sph$. Appendix \ref{app:pflaum} briefly summarizes the construction in \cite{pflaum}.
Appendix \ref{se:app} contains a few remarks about $*$-structures for twisted tensor products.  \\

\noindent
\textbf{Acknowledgments.}
We are pleased to thank Giovanni Landi for suggesting us the problem addressed here  and for fruitful discussions and Richard Szabo for useful comments. LSC thanks SISSA (Trieste) and UCL Louvain for scientific invitations during the development of this research project. Part of the work was carried out at the University of Luxembourg, both of us acknowledge support by the National Research Fund, Luxembourg and COFUND MarieCurie. 

\section{The quantum 4-sphere and its local charts}\label{sect:4sphere}

In this section we recall  the description of the noncommutative four-sphere $\sph$ as from  \cite[\S 3.5]{clsII}. The algebra $\sph$ is a one real parameter deformation of the algebra of coordinate functions on the classical four-sphere, where this latter is seen as a real subspace of the Klein quadric in $\M{CP}^5$.
In \cite[\S 5.3]{clsI} the authors define for each couple of integers $d < n$ a noncommutative Grassmannian  $\A(\mathbb{G}r_\theta(d,n))$ as a quotient of a projective space $\A(\mathbb{CP}^N_\theta)$, $N=\footnotesize{\left(\begin{array}{c}n\\d \end{array} \right)}-1$. The commutation relations among the generators of these algebras depend on an $n \times n$ matrix $\Theta$ of deformation parameters and are derived from a general theory of cocycle (or Drinfeld twists) Hopf deformations. For $d=2$, $n=4$ and a particular choice of the matrix $\Theta$ which reduces the parameters from six complex $\theta_{ij}$ to a pure imaginary one $\theta$, or equivalently to a real
one $q=\mathrm{exp}(\mathrm{i}\, \theta / 2) \, \in \IR$, they obtain  the noncommutative Grassmannian $\A(\mathbb{G}r_q(2,4))$.  The algebra $\A(\mathbb{G}r_q(2,4))$ is generated by 
 `Pl\"ucker coordinates' $\La{ij}$, $i<j \; (i,j=1,\dots, 4),$ subject to the following commutation relations

\be\label{comm_lambda}
\begin{array}{lll}
\La{12}\La{13}= q^{-2} \La{13}\La{12} \quad, \quad &\La{12}\La{14}= q^{-2} \La{14}\La{12} \quad, \quad 
& \La{12}\La{23}= q^{2} \La{23}\La{12} \quad, \vspace{3pt}
\\ 
\La{12}\La{24}= q^{2} \La{24}\La{12} \quad, \quad &\La{12}\La{34}=  \La{34}\La{12} \quad, \quad 
& \La{13}\La{14}=  \La{14}\La{13} \quad, \vspace{3pt}
\\ 
\La{13}\La{23}= q^{2} \La{23}\La{13} \quad, \quad &\La{13}\La{24}= q^2 \La{24}\La{13} \quad, \quad 
& \La{13}\La{34}=  \La{34}\La{13} \quad,\vspace{3pt}
\\ 
\La{14}\La{23}= q^{2} \La{23}\La{14} \quad, \quad &\La{14}\La{24}= q^2 \La{24}\La{14} \quad, \quad 
& \La{14}\La{34}=  \La{34}\La{14} \quad,\vspace{3pt}
\\ 
\La{23}\La{24}=  \La{24}\La{23} \quad, \quad &\La{23}\La{34}= \La{34}\La{23} \quad, \quad 
& \La{24}\La{34}=  \La{34}\La{24} \quad 
\end{array}
\ee
and 
satisfying the Klein quadratic identity
\be\label{quadric_gr}
q\La{12}\La{34}- \La{13} \La{24} + \La{14} \La{23}=0 \; .
\ee
We introduce a $*$-structure on $\A(\mathbb{G}r_q(2,4))$ by defining it on the generators as \footnote{The $*$-structure in \cite{clsII} was defined slightly differently with $\La{14}^* = - q^{-1}\La{23}$.}  
 \be\label{*}
\La{12}^*=  \La{12} \; , \quad \La{13}^*= q \La{24} \; , \quad  \La{14}^*=-q \La{23} \;, 
\quad
  \La{34}^*= \La{34}
\ee
and extending it to the whole algebra as an anti-algebra morphism.
The Klein identity \eqref{quadric_gr} has real form
\be\label{quadric}
q\La{12}\La{34}- q^{-1}\La{13} \La{13}^* -q^{-1} \La{14} \La{14}^*=0 \;
\ee
which is readily shown to correspond to signature $(5,1)$ by considering the change of generators $X:=\frac12 q(\La{12}-\La{34})$ and $R:=\frac12 q(\La{12}+\La{34})$, thus giving
\be\label{eq_sfera}
X^2 +\La{13} \La{13}^* + \La{14} \La{14}^*=R^2  \, .
\ee
\begin{defi}
The coordinate algebra $\sph$ is the $*$-algebra generated by elements $\La{ij}$, $i<j$, $i,j=1,\dots, 4$ satisfying the commutation relation \eqref{comm_lambda} and 
the quadratic identity \eqref{quadric}, and with 
$*$-structure as in
\eqref{*}.
\end{defi}

$\sph$ is a one-parameter deformation of a 4-sphere described in homogeneous coordinates.
We observe that while the deformation parameter $q$ enters the commutation  relations \eqref{comm_lambda} among the generators of $\sph$, the sphere relation  \eqref{eq_sfera} is classical. Nevertheless we remark that the `radius' $R$ of $\sph$ is noncommutative, namely the generator $R$ does not belong to the center of the algebra.  Moreover $R$ does not even generate a left or right denominator set (in the sense of Ore localization), so that it is not possible to localize with respect to it and obtain a `global' affine description of the sphere. As discussed in the introduction,  it follows the need of using a local (sheaf-theoretic) approach  to construct a Hopf bundle on  $\sph$.
 
We introduce two `local patches'  $\rn$, $\rs$ of the sphere $\sph$ which are 
obtained via Ore localization with respect to the two real generators $\La{34}$ and $\La{12}$ respectively.  Such localizations provide quantum analogue of stereographic projections from the North and South poles; we refer to the original paper \cite{clsII} for details. Let us stress a substantial difference between the two generators: while $\La{34}$ is central and hence the Ore localization reduces to the standard commutative localization, $\La{12}$ is not.  

\begin{prop}\cite[Prop.3.19]{clsII}
The degree zero subalgebra $\left(\A(\mathbb{G}r_\theta(2,4)) [\La{34}^{-1}]\right)_0$ of the \textbf{right} Ore localization of $\A(\mathbb{G}r_\theta(2,4))$ with respect to $\La{34}$ is isomorphic to the algebra 
generated by elements $\b{13},~\b{14},~\b{23}$ and $\b{24}$ subject to the commutation relations  
\be
\begin{array}{lll}
\label{ncbeta}
\b{13}\b{14}= \b{14}\b{13} \; , \quad & \b{13}\b{23}= q^2 \b{23} \b{13} \; , \quad &\b{13}\b{24}= q^2 \b{24} \b{13} \; , \vspace{3pt}
\\
\b{14}\b{23}= q^2 \b{23}\b{14} \; ,  & \b{14}\b{24}= q^2 \b{24} \b{14} \;, & \b{23}\b{24}= \b{24} \b{23} \; .
\end{array} 
\ee
\end{prop}

\noindent
The proof (that we avoid to recopy here) is a direct computation once one makes the identification $\b{ij} = \La{ij}\La{34}^{-1}$. Note that the generator $\beta:= \La{12} \La{34}^{-1}$ is redundant; indeed  the Pl\"ucker relation \eqref{quadric_gr} implies 
\be
\label{ro}
\beta= q ( \b{24}\b{13}-  \b{23}\b{14}) = q^{-1} (\b{13}\b{24} - \b{14}\b{23} ) \, .
\ee
The $*$-structure defined  in \eqref{*} induces on the generators $\b{ij}$ the relations
\be\label{*beta}
\b{24}^*= q^{-1} \b{13} \; , \quad  \b{23}^*=-q^{-1} \b{14} \, .
\ee
The resulting $*$-algebra will be denoted by $\rn$. The above equation \eqref{ro} becomes 
\be 
\label{ros}
\beta = \b{13}^*\b{13} + \b{14}^*\b{14} = \b{23}^*\b{23} + \b{24}^*\b{24} \, .
\ee

\begin{prop}\cite[Prop.3.22]{clsII}
The degree zero subalgebra $\left([\La{12}^{-1}]\A(\mathbb{G}r_\theta(2,4))\right)_0$ of the \textbf{left} Ore localization of $\A(\mathbb{G}r_\theta(2,4))$ with respect to $\La{12}$ is isomorphic to the algebra 
generated by elements $\a{13},~\a{14},~\a{23}$ and $\a{24}$ subject to the commutation relations  
\be
\begin{array}{lll}
\a{13}\a{14}= \a{14}\a{13} \; , \quad & \a{13}\a{23}= q^{-2} \a{23} \a{13} \; , \quad &\a{13}\a{24}= q^{-2} \a{24} \a{13} \; , \vspace{3pt}
\\
\a{14}\a{23}= q^{-2} \a{23}\a{14} \; ,  & \a{14}\a{24}= q^{-2} \a{24} \a{14} \;, & \a{23}\a{24}= \a{24} \a{23} \; .
\end{array} 
\ee
\end{prop}

\noindent
Similarly to the previous case, one gets the relations above by putting $\a{ij}=\La{12}^{-1}\La{ij}$; by using \eqref{quadric_gr} the element $\alpha:= \La{12}^{-1} \La{34} $ can be expressed in terms of the $\a{ij}$ as 
\be
\label{roi}
\alpha= q^{-1} ( \a{24}\a{13}-  \a{23}\a{14}) = q (\a{13}\a{24} - \a{14}\a{23} ) \, .
\ee
We denote by $\rs$ the $*$-algebra generated by the $\a{ij}$ and endowed with the $*$-structure 
\be\label{*alpha}
\a{24}^*= q^{-3} \a{13} \; , \quad  \a{23}^*=-q^{-3} \a{14} \, .
\ee
induced from \eqref{*}. The equation \eqref{roi} becomes
$$ \alpha = q^{-4} ( \a{13}^*\a{13} + \a{14}^*\a{14}) = q^4 (\a{23}^*\a{23} + \a{24}^*\a{24} ) \, .
$$

\begin{prop}\label{prop:G} 
There exists a $\, *$-algebra isomorphism $\Q:\rs\rightarrow\rn$ defined on generators as
\be
\label{mapQ}
 \Q(\a{13})=  q^{-2} \b{13} 
 , \quad
 \Q(\a{24})=  q^{2} \b{24} 
 , \quad
 \Q(\a{14} )= q^{-2} \b{14} 
 , \quad
 \Q(\a{23})=  q^{2} \b{23} 
 , \quad
 \Q(q)=  q^{-1} \, .
\ee
\end{prop}

The proof is a direct computation and we omit it. A geometric interpretation of this isomorphism is discussed in Remark \ref{commQ}.\\

The intersection of the two `charts'  $\rn$ and $\rs$ is geometrically obtained by removing the `origin' in each patch. Algebraically, this amounts to extending $\rn$ with the inverse of $\beta$, i.e. to introducing an extra generator $\beta^{-1}$ with commutation relations
\be
\b{13}\beta^{-1}= q^{-2}\beta^{-1} \b{13} \; , \quad 
\b{14}\beta^{-1}= q^{-2} \beta^{-1}\b{14} \; , \quad
\b{23}\beta^{-1}= q^{2} \beta^{-1}\b{23} \; , \quad
\b{24} \beta^{-1}= q^{2} \beta^{-1}\b{24}  
\ee
together with 
\be
\label{roinv}
\beta^{-1}(\b{13}\b{24}- \b{14}\b{23}) = q = (\b{13}\b{24}- \b{14}\b{23}) \beta^{-1} \, .
\ee
We denote by $\rsn$ the extension of $\rn$ generated by adjoining the element  $\beta^{-1}$. Similarly, we can extend $\rs$ with the inverse of $\alpha$ by introducing an extra generator $\al^{-1}$ with commutation relations 
\be
\a{13}\al^{-1}= q^{2}\al^{-1}\a{13} \; , \quad 
\a{14}\al^{-1}= q^{2}\al^{-1}\a{14} \; , \quad
\a{23}\al^{-1}= q^{-2}\al^{-1}\a{23} \; , \quad
\a{24}\al^{-1}= q^{-2}\al^{-1}\a{24}  
\ee
together with  
\be
\label{roiinv}
\al^{-1}(\a{24}\a{13}- \a{23}\a{14}) = q =  (\a{24}\a{13}- \a{23}\a{14}) \al^{-1} \, .
\ee
We denote the resulting algebra by $\rns$. The above two descriptions of the intersection are equivalent as a consequence of Proposition \ref{prop:G}:
\begin{cor} 
\label{cor:G} 
The map $\Q$ extends to a 
$\, *$-algebra isomorphism $\Q:\rns\rightarrow\rsn$ with $\Q(\alpha)=  \beta$. 
\end{cor}
An algebra isomorphism between $\rns$ and $\rsn$ was already observed in  \cite{clsII}, but realized there via a different map  $G: \rns \rightarrow \rsn$ which was defined on the generators as $G(\a{ij})= \al \b{ij}$, $G(\ali)   = \alpha^{-1}$ and $G(q)=q$. The different choice of the isomorphism $\Q$ made here is relevant to what follows.
\medskip

\section{The sheaf of the quantum $4$-sphere}\label{sec-sheaf}

In this section we construct a sheaf of noncommutative algebras on the classical 
four-sphere $S^4$. This is the structure sheaf describing a \textit{quantum space}, in the terminology of \cite[\S 2]{pflaum}, and it is obtained by considering suitable algebras constructed out of the local charts of the noncommutative $4$-sphere $\sph$. We also provide an analogous  sheaf-theoretic realization of the noncommutative differential calculus of $\sph$. 

We begin with a new description of the intersection of the charts in terms of \textit{twisted tensor product} of  algebras. We refer to Appendix \ref{se:app} for the terminology and the few results of the theory of twisted tensor products that will be used. The convenience of using this characterization of the intersection of the charts  will be manifest in the construction of the quantum Hopf bundle in \S \ref{sec-Hopf}. 

\subsection{The intersection of charts}

The intersection of the two charts $\rn$, $\rs$ of $\sph$ admits a very natural geometric interpretation if we write the real and positive generator $\beta^{-1}$ of $\rsn$ as $\beta^{-1}=r^{-2}$.  In Prop. \ref{prop:su2beta} and  Prop. \ref{prop:su2alpha} below we show  that by suitably rescaling the generators of $\rsn$ by $r^{-1}$ (or equivalently those of $\rsn$ by $r$)  we can recover a classical $SU(2)\simeq S^3$ inside the intersection of the two charts.

\begin{defi}
We denote by $\sqrsn$ the $*$-algebra generated by elements $\{ \b{13},\b{14},\b{23},\b{24}, r^{-1} \}$  with commutation relations
\begin{equation}
\label{sqrsncr}
\b{13} r^{-1} = q^{-1} r^{-1}\b{13} \; , \quad \b{14} r^{-1} = q^{-1} r^{-1}\b{14} \; , \quad
\b{23} r^{-1} = q r^{-1}\b{23} \; , \quad \b{24} r^{-1} = q r^{-1}\b{24} \; , \quad
\end{equation} 
and satisfying
\begin{equation}
\label{sqrsnrel}
r^{-2}(\b{13}\b{24} - \b{14}\b{23}) = q = (\b{13}\b{24} - \b{14}\b{23}) r^{-2} \, .
\end{equation}
The $*$-structure is the one given in \eqref{*beta}, together with $(r^{-1})^*=r^{-1}$.
\end{defi}
\noindent

From \eqref{sqrsnrel} we can deduce that $r^{-2}$ is invertible in $\sqrsn$ with inverse (denoted $r^2$) given by 
\begin{equation}
\label{r2}
r^2:=q^{-1}(\b{13}\b{24} -\b{14}\b{23}) = \b{24}^*\b{24} + \b{23}^*\b{23} \, .
\end{equation}
Thus $r^{-1}$ is invertible as well: denoting with $r$ its inverse, the explicit expression is $r= r^2 r^{-1} = r^{-1} r^2$. We note that $r^2$ is a combination of the $\b{ij}$'s alone, contrary to $r$ which 
contains a contribution from $r^{-1}$ as well.

The $*$-algebras $\rsn$ and $\sqrsn$ are not isomorphic, nevertheless they describe the same `quantum space' in the sense that they generate the same $C^*$-algebra. Indeed at the $C^*$-algebra level the element $\alpha^{-1} = \b{24}^* \b{24} + \b{23}^* \b{23}$ is positive being the sum of the two positive elements $\b{24}^* \b{24}$ and  $\b{23}^* \b{23}$ (see e.g. \cite[Thm. I.4.5, Cor. I.4.4]{dav}). Therefore by the general theory there exists  a (unique) element $r=r^*\in \rsn$ such that   $\alpha^{-1}=r^2$, (see e.g. \cite[Cor. I.4.1]{dav}).

\begin{prop}\label{prop:su2beta}
Let $\A$ be the  $*$-subalgebra of $\sqrsn$ generated by the elements
\be
\label{gensu2}
\x{23}:= \b{23}r^{-1} \, , \quad \x{24}:=\b{24}r^{-1} \, , 
\quad \x{23}^*=r^{-1}\b{23}^* \, , \quad \x{24}^* =r^{-1}\b{24}^* \, . 
\ee
Then   the sphere relation
$$ \x{23}^* \x{23} + \x{24}^* \x{24} = 1 \,  $$
holds 
and $\A$ is commutative.
\end{prop}
\proof{The commutativity among the generators can be checked by a direct computation. The
sphere relation is  obtained by multiplying eq. \eqref{r2} from both sides by $r^{-1}$. \qed} \\

\noindent
For future convenience, let us  organise the generators of the algebra $\A$ as entries of a $2 \times 2$ matrix
\be\label{matrixA}
A:=
\begin{pmatrix} 
\phantom{-}\x{23} & \x{24} \; \\
-\x{24}^* & \x{23}^* \;    \end{pmatrix}  
\ee
with $\det(A) =\x{23} \x{23}^* + \x{24} \x{24}^*=1$.
\begin{prop}
\label{hopf}
The $\, *$-algebra $\A$ can be endowed with a Hopf algebra structure by defining the coproduct, counit and antipode on the generators as
$$\Delta(A)=A \overset{.}{\otimes} A~, \quad\varepsilon(A)=\M{I} ~, \quad
S(A)= \begin{pmatrix}
\, \x{23}^*  & - \x{24} \; \\
\, \x{24}^* & \phantom{-}\x{23} \; 
\end{pmatrix} \, .
$$
The resulting Hopf algebra is the coordinate Hopf algebra $\aa(SU(2))$ of $SU(2)$. 
\end{prop}

We now show that $\sqrsn$ can be factorized  into a product of a commutative $3$-sphere $S^3\simeq SU(2)$ represented by $\aa$ and  of a $1$-dimensional interval $I$, algebraically described as the $*$-algebra $\I$ generated by $\{r,r^{-1}\}$ satisfying the relation $rr^{-1}=r^{-1}r=1$. While both $\aa$ and $\I$ are commutative, the noncommutativity emerges from their tensor product. The factorization is expressed in Proposition \ref{prop:fact} below. This is an elementary result from the algebraic point of view, nevertheless it provides a very nice geometrical picture. 
\medskip

\begin{defi}
We denote by $\ai$ the twisted tensor product algebra consisting of the vector space $\aa \ot \I$ endowed with the multiplication 
$$
m_\theta :=(m_\aa \ot m_\I)(\id_\aa \ot \Psi \ot \id_\I)
$$
where $\Psi : \I \ot \aa \rightarrow \aa \ot \I$ is the linear map defined on the vector space base elements by
\be\label{twist}
\Psi \left( r^{\pm n} \ot \x{23}^a ~ ( \x{23}^*)^b ~ \x{24}^c ~ (\x{24}^*)^d \right) :=
 q^{\pm n (a+c) \mp n(b+d)} ~ \x{23}^a ~ ( \x{23}^*)^b ~\x{24}^c ~ (\x{24}^*)^d \ot r^{\pm n}
\ee 
for all integers $n,a,b,c,d  \in \mathbb{N} \cup \{ 0\}$. 
\end{defi}

\noindent
We remark that the twist $\Psi$ is normal: $\Psi(1 \ot x) = x \ot 1$ and $\Psi(r^{\pm a} \ot 1) = 1 \ot r^{\pm a}$, $\forall x \in \aa$, $a \in \mathbb{N}$. Notice that $\Psi $ is \textit{not} an algebra morphism. On the algebra generators it reads  
\begin{eqnarray}\label{twist-gen}
\Psi(r^{\pm 1} \ot \x{23}) = q^{\pm 1} \x{23} \ot r^{\pm 1}   \quad &;& \quad    \Psi(r^{\pm 1} \ot \x{23}^*) = q^{\mp 1} \x{23}^* \ot r^{\pm 1}
\nn \\
\Psi(r^{\pm 1} \ot \x{24}) = q^{\pm 1} \x{24} \ot r^{\pm 1}  \quad &;& \quad    \Psi(r^{\pm 1} \ot \x{24}^*) = q^{\mp 1} \x{24}^* \ot r^{\pm 1} \, .
\end{eqnarray} 

\begin{lem}
The algebra $\ai$ is associative and unital.
\end{lem}
\begin{proof} 
From the general theory of twisted tensor product algebras, in order to prove that the multiplication $m_\theta$ is associative it is enough to show that the normal twist $\Psi$ satisfies the following two conditions (see e.g. \cite{cae} and the Appendix):
\begin{eqnarray}
&&(\id_\aa\ot m_\I)(\Psi \ot \id_\I)(\id_\I \ot \Psi) =\Psi (m_\I \ot \id_\aa )
\, , \label{P1}
\\
&& (m_\aa \ot \id_\I)(\id_\aa\ot\Psi)(\Psi \ot \id_\aa) =\Psi (\id_\I \ot m_\aa )\,  .\label{P2}
\end{eqnarray}
These are easily verified  by using the explicit form of the twist as given in \eqref{twist} above. 
Finally, $1 \ot 1$ is the unit in $\ai$ from general results of the theory of normal twists. 
\end{proof}
\begin{lem}
The map $\Psi'=\Psi \circ \tau$ is $*$-compatible. Then  $\ai$ is a $*$-algebra with involution 
 $$ *_\Psi (x \ot j) := (x \ot j)^{*_\Psi}:=  \Psi(j^* \ot x^*), \forall x \in \aa, j\in \I \, .$$ 
\end{lem}
\begin{proof}
By direct check and  direct application of Prop. \ref{prop:*}. 
\end{proof}
\noindent
On the generators of $\ai$ the $*$-structure  reads
\be\label{*AI}
(\x{23} \ot r^{\pm 1})^{*_\Psi}= q^{\mp 1} \x{23}^* \ot r^{\pm1} \quad ; \quad
(\x{24} \ot r^{\pm 1})^{*_\Psi}= q^{\mp 1} \x{24}^* \ot r^{\pm1} \, .
\ee

\begin{prop}\label{prop:fact}
The map $f_{SN}:\sqrsn \rightarrow \ai$  defined on the generators as
\begin{equation}
\label{deff}
\fsn (\b{23}) = x_{23} \ot r \; , \quad \fsn (\b{24}) = x_{24}\ot r \; , \quad \fsn (r^{-1}) =1 \ot r^{-1}
\end{equation}
and extended as a $*$-algebra morphism is an isomorphism of unital $*$-algebras.
\end{prop}
\begin{proof}
We only have to prove that the map $\fsn$ preserves the commutation relations \eqref{sqrsncr} among the generators and  equation \eqref{sqrsnrel}. We verify the latter; the other identities are proved in a similar way and we omit to transcribe the proof. We have
\begin{eqnarray*}
\fsn(r^{-2}) \fsn (\b{24}^*\b{24} + \b{23}^*\b{23}) &=& (1 \ot r^{-2}) \left[ q^{-1} (\x{24}^* \ot r)   (\x{24} \ot r)+
q^{-1} (\x{23}^* \ot r)   (\x{23} \ot r) \right]  
\\ &=& q \left[ (\x{24}^* \ot r^{-1})   (\x{24} \ot r)+
(\x{23}^* \ot r^{-1})   (\x{23} \ot r) \right] \\ &=&    (\x{23}^* \x{23} + \x{24}^* \x{24})  \ot 1 = 1 \ot 1 \, 
\end{eqnarray*}
where we made use of \eqref{*AI}. This also proves that $\fsn$ preserves the units. 
\end{proof}
\medskip
\begin{rem} 
It is possible to get a commutative subalgebra in the intersection of the two charts also without rescaling by $r^{-1}$. Consider the subalgebra $\aa'\subset\rns$ generated by 
\be
a_{13}:= \a{13} \; , \quad 
a_{14}:= \a{14} \; , \quad 
a_{23}:=  q^{-1} \ali\a{23} \; , \quad 
a_{24}:=  q \a{24} \ali \; .
\ee
We have the identity
\be\label{eq_det}
a_{24} a_{13}-a_{23}a_{14} =1
\ee
and $\aa'$ is commutative. Therefore $\aa'$ can be endowed with a Hopf algebra structure as done for $\aa$ in Proposition \ref{hopf}, thus obtaining the coordinate Hopf algebra of $SL(2)$. But note that $\aa'$ is not a $\,*$-subalgebra of $\rns$: from \eqref{*alpha}  we have
\be
a_{23}^*= -q^{-4} a_{14} \ali \quad , \quad a_{24}^*= q^{-2}\ali a_{13} 
\ee
and the identity \eqref{eq_det} becomes 
\be\label{3-sphere}
q^2 a_{23}^* a_{23} + q^2 a_{24}^* a_{24} = \ali
\ee
hence describing the equation of a three-sphere with invertible but non central radius. To recover the coordinate Hopf algebra of $SU(2)$ we would have to introduce on $\aa'$ a different (i.e. not inherited from $\rns$) $\, *$-structure, namely
$$
\bar{a}_{24} :=a_{13} \; , \quad \bar{a}_{23} :=a_{14} \, .
$$
\end{rem}
\bigskip

We can, of course, obtain analogous results by using the algebra $\rns$. We omit the proofs. 
\begin{defi}
We denote by $\sqrns$ the $*$-algebra generated by $\{ \a{13},\a{14},\a{23},\a{24}, r \}$, together with commutation relations
\begin{equation}
\label{sqrnscr}
\a{13} r = q ~r \a{13} \; , \quad \a{14} r = q~ r \a{14} \; , \quad
\a{23} r = q^{-1} ~r \a{23} \; , \quad \a{24} r = q^{-1} r \a{24} \; , \quad
\end{equation} 
and satisfying
\begin{equation}
\label{sqrnsrel}
r^{2}(\a{24}\a{13} - \a{23}\a{14}) = q = (\a{24}\a{13} - \a{23}\a{14}) r^{2} \, .
\end{equation}
The $*$-structure is the one given in \eqref{*alpha}, together with $r^*=r$.
\end{defi}

\noindent
The inverse of $r^{2}$ in $\sqrns$, denoted $r^{-2}$, is computed from \eqref{sqrsnrel}:  
\begin{equation}
\label{r-2}
r^{-2}:=q^{-1}(\a{24}\a{13} -\a{23}\a{14}) \, .
\end{equation}
Also $r$ is invertible, with ${r}^{-1}= r^{-2} r = r r^{-2}$.  The map $\Q$ (see \eqref{mapQ} and
Corollary \ref{cor:G}) induces a $\, *$-algebra isomorphism $\widetilde{\Q}:\sqrns\rightarrow\sqrsn$ with $\widetilde{\Q}(r^{-1})=  r$.

\begin{prop}\label{prop:su2alpha}
The subalgebra of $\sqrns$ generated by
\be
\label{gensu2y}
\y{23}:= q~r~ \a{23} \, , \quad \y{24}:=q~r~\a{24} \, , \quad \y{23}^*= q \a{23}^* r \, , 
\quad \y{24}^* =q \a{24}^*r  
\ee
is commutative and coincides with $\widetilde{\Q}^{-1}(\A)$. Moreover, the following sphere relation holds:
$$ \y{23}^* \y{23} + \y{24}^* \y{24} = 1 \, . $$
\end{prop}

\begin{prop}\label{prop:facty}
The map $\fns:\sqrns \rightarrow \ai$ defined on the generators as
\begin{equation}
\label{deffy}
\fns(\a{23}) = q^{-2} x_{23} \ot r^{-1} \; , \quad \fns(\a{24}) = q^{-2} x_{24}\ot r^{-1} \; , \quad \fns(r) =1 \ot r
\end{equation}
and extended as a $*$-algebra morphism is an isomorphism of unital $*$-algebras.
\end{prop}
\noindent
We conclude by noting that from the above results the diagram of $*$-algebra isomorphisms  below is commutative:
\be
\xymatrix{
\sqrns \ar[rr]^{\widetilde{\Q}} \ar[dr]_{\fns} & & \sqrsn \ar[dl]^{\fsn}  \\
& \ai & }
\ee

\subsection{The structure sheaf $\mathcal{\O}_{S^4_q}$ and the differential calculus sheaf $\dfs$}

In classical geometry spaces can be characterized by their structure sheaf. More precisely, a 
topological space $M$ together with a sheaf of commutative rings $\O_M$ on $M$, referred to as the structure sheaf of $M$, form a \textit{ringed space} $(M,\O_M)$. By specifying a local model for the sheaf $\O_M$ we recover different geometrical notions. Consider for example the ringed spaces $(\M{R}^n,\mathcal{C})$ and $(\M{R}^n,\mathcal{C}^m)$, where $\mathcal{C}$ is the  sheaf of continuous functions on $\M{R}^n$, $n\in\M{N}$, and $\mathcal{C}^m$, $m\in\M{N}\cup \{\infty\}$, is the sheaf of $m$-times differentiable functions on $\M{R}^n$. Then, if $(M,\O_M)$ is locally isomorphic as a ringed space to $(\M{R}^n,\mathcal{C})$ (resp. $(\M{R}^n,\mathcal{C}^m)$) we say that $M$ is a topological manifold (resp. differentiable $m$-manifold) of dimension $n$. Further different choices of the local model characterize analytical manifolds, complex manifolds, schemes \cite[Ex. 2.4]{pflaum}. Following \cite{pflaum}, we call a \textit{quantum space} over $M$ a ringed space $(M,\O_M)$ where $\O_M$ is now a sheaf of (not necessarily commutative) algebras. In this framework the `local charts' of the noncommutative algebra $\sph$ naturally define a quantum space over the classical $4$-sphere $S^4$, as we are going to show.

Let us consider the topology of $S^4$ whose basis consists of the following sets: $U_N:=S^4 \backslash \{NP\}$, $U_S:=S^4 \backslash \{SP\}$ and their intersection $U_{SN}:=S^4 \backslash \{NP, SP\}$, where $NP$ and $SP$ denote the North and South poles respectively. 
We construct a sheaf $\shs$ of noncommutative $*$-algebras on  $S^4$ by the assignment

\begin{equation}\label{sheafB}
\shs(U_N) := \rn \, , \qquad
\shs(U_S) := \rs \, , \qquad
\shs(U_{SN}) := \sqrsn \simeq  \ai
\end{equation}
together with restriction maps
\begin{eqnarray} \label{restr}
\rho_{N,SN} & :   \rn  \ra  \sqrsn \, , \qquad \beta_{ij}& \mapsto \beta_{ij}\nn \\
\rho_{S,SN} & :   \rs  \ra  \sqrsn \, , \qquad \alpha_{ij}&\mapsto   \widetilde{\Q}(\a{ij})
\end{eqnarray}
and identities otherwise. We observe that $\shs$ is not a flabby sheaf, i.e. its restriction maps are not surjective. By standard arguments in sheaf theory, $\shs$ is defined on the whole topology of $S^4$; in particular 
\begin{eqnarray}\label{restrB}
\shs(S^4)&=& \shs(U_N \cup U_S) \nn \\&=& \{(a_N,a_S,a_{SN}) \in \shs(U_N) \oplus \shs(U_S) \oplus \shs(U_{SN})~|~  \rho_{N,SN}(a_N)=
\rho_{S,SN}(a_S)=a_{SN}
 \} \nn
\\ \label{glsshs}
&\simeq&
\{(a_N,a_S) \in \shs(U_N) \oplus \shs(U_S) ~|~  \rho_{N,SN}(a_N)=
\rho_{S,SN}(a_S)
 \}.
\end{eqnarray}

\begin{defi}
The quantum space of the noncommutative 4-sphere $S^4_q$ is the sheaf of noncommutative algebras $\shs$ over the classical 4-sphere $S^4$. 
\end{defi}

\begin{rem}
\label{commQ}
The isomorphism $\Q$ of Proposition \ref{prop:G} ensures that the quantum space $S^4_q$ has  the same quantum deformation $\M{R}^4_q:=\!\!\rn\simeq\!\!\rs$ of $\M{R}^4$ as local model, i.e. the quantum sphere is locally isomorphic to (isomorphic copies of) $\M{R}^4_q$. Quantum spaces generally lack a unique local model in the sense that they may host non-isomorphic noncommutative deformations in different patches. For example the twistor bundle over $S^4_q$  constructed in \cite[\S 4.6]{clsII} has local trivializations to $\M{R}^4_q\otimes_q\M{CP}^1_q$ and to $\M{R}^4_q\otimes \M{CP}^1$ depending on the two patches.   
\end{rem}\medskip

We now describe the sheaf $\dfs$ of differential forms on
$S^4_q$. In classical geometry differential forms admit a natural description in terms of sheaves, and so do their noncommutative deformations.  The noncommutative differential calculus of $S^4_q$ was originally constructed in \cite[\S 6.2]{clsII} in terms of forms defined on the two local patches $\rn$ and $\rs$ of the quantum sphere and on their intersection. The calculus was obtained via Drinfeld twist deformation of the canonical differential calculi on $\IR^4$ by using the 2-cocycle which controls the noncommutativity of $S^4_q$ (see the Introduction). Here we describe the noncommutative differential calculus via the data of a sheaf $\dfs$ of graded noncommutative algebras. We  state the relevant noncommutative relations among the generators while omitting any reference to the Drinfeld twist behind them. 

Let us start with the open set $U_N\subset S^4$. The structure sheaf reads $\shs(U_N)=\rn$
and the algebra $\rn$ admits a canonical noncommutative differential $*$-calculus $\ndf=\wedge^{\bullet}_q {_N\Omega^1_q}$. The latter is a differential graded $*$-algebra whose degree zero part coincides with $\rn$ and whose degree one part has generators $\{ \d\b{23}, \, (\d\b{23})^*, \, \d\b{24}, \, (\d\b{24})^* \, \}$. The differential $\d:\rn\rightarrow {_N\Omega^1_q}$ is defined as
$$ \d(\b{23}) = \d\b{23} \, , \quad \d(\b{23}^*)=(\d\b{23})^*  \, , \quad
 \d(\b{24})= \d\b{24} \, , \quad  \d(\b{24}^*)= (\d\b{24})^* $$
and uniquely extended to a degree one $*$-map $\d:{_N\Omega^m_q}\rightarrow {_N\Omega^{m+1}_q}$ by requiring it to be  $\mathbb{C}$-linear and to satisfy $\d^2=0$ together with the Leibniz rule
\begin{equation}
\label{qleib}
\d (\omega \wedge_q \omega') = \d\omega \wedge_q\omega' + (-1)^{|\omega|} \omega \wedge_q \d\omega'
\, .
\end{equation}
Since there is no risk of confusion, we will simply write $\d\b{23}^*$, resp. $\d\b{24}^* $ to indicate
$(\d\b{23})^*$, resp.  $(\d\b{24})^*$. To simplify further the notation we often drop the symbol $\wedge_q$ in formulas.

The commutation relations in $\ndf$ are completely derived from those involving only elements of degree zero and one.  In addition to the relations in $\rn$ we have
\begin{align*}
\b{23}\d\b{23} & = \d\b{23}\b{23} \, ,& \b{23}(\d\b{23})^* & = q^{-2} (\d\b{23})^* \b{23} \, ,& 
\b{23}\d\b{24} & = \d\b{24}\b{23} \, ,& \b{23}(\d\b{24})^* & = q^{-2} (\d\b{24})^* \b{23} \, ,\\
\b{24}\d\b{24} & = \d\b{24}\b{24} \, ,& \b{24}(\d\b{24})^* & = q^{-2} (\d\b{24})^* \b{24} \, ,& 
\b{24}\d\b{23} & = \d\b{23}\b{24} \, ,& \b{24}(\d\b{23})^* & = q^{-2} (\d\b{23})^* \b{24} \, . 
\end{align*}      
By applying the maps $*$ and $\d$ one gets the remaining relations in degree one and from there the ones in higher degrees. \\

A completely analogous construction is carried out in the second chart $U_S\subset S^4$. The algebra $\rs$ admits a canonical noncommutative differential calculus $\sdf=\wedge_q^{\bullet}{_S\Omega^1_q}$ built out of $_S\Omega^0_q:=\!\!\rs$ and a differential $\d:\!\!\rs\rightarrow {_S\Omega^1_q}$, $\d(\a{ij})\mapsto\d\a{ij}$ which are constructed by following the same prescriptions as above. One can equivalently extend the $*$-isomorphism $\Q:\rs\rightarrow\rn$ to 1-forms by putting $q^{2}\d\a{23}=\Q^{-1}(\d\b{23}):=\d\Q^{-1}(\b{23})$ and so on (the $q$-coefficients come from the definition of $\Q$ in \eqref{mapQ}), hence realizing the commutative diagram
\begin{equation}
\label{cdomega}
\xymatrix{
\rs \ar[r]^{\Q} \ar[d]_{\d} & \rn \ar[d]^{\d} \\
{_S\Omega^1_q} \ar[r]_{\Q} & {_N\Omega^1_q} 
}
\end{equation} 
where ${_S\Omega^1_q}:=\Q^{-1}(_N\Omega^1_q)$. Then the isomorphism extends to the whole differential calculus, $\sdf = \Q^{-1}(\ndf)$.

In the intersection $U_{SN}$ we consider the noncommutative differential calculus $\sndf$ of the algebra $\sqrsn$. Again this can be done by following the standard construction or more directly by requiring that the process of algebraic extension commutes with the differential. This amounts to define $\sndfu$ as the extension of ${_N\Omega^1}_q$ by an extra-generator $\d r^{-1}=(\d r^{-1})^*$, such that the additional commutation relations involving $\d r^{-1}$ are

\begin{align*}
 r\d r^{-1} &= \d r^{-1}r \,  \quad
 &\b{23}\d r^{-1} &= q^2 \d r^{-1}\b{23}
 \, , \quad
 &\b{24}\d r^{-1}&= q^2 \d r^{-1}\b{24} \, , 
\\
\b{23}^*\d r^{-1}  &= q^{-2} \d r^{-1}\d\b{23}^* \, , \quad
&
 \b{24}^*\d r^{-1} &= q^{-2} \d r^{-1}\d\b{24}^* \, , \quad
&& \, 
\end{align*}
together with those obtained from them by applying the maps $\d$ and $*$. Finally, we have the differential counterpart of the sphere relation $r^{-2}(\b{23}^*\b{23}+\b{24}^*\b{24})=1$ (see \eqref{r2}) which reads
\begin{equation}
\label{dr2}
r^{-1}\d r^{-1} + (\d r^{-1}) r^{-1} + (\d\b{23})^*\b{23} + \b{23}^*\d\b{23} + (\d\b{24})^*\b{24} + \b{24}^*\d\b{24} = 0 \, .
\end{equation} \vspace{3pt}

We arrange these noncommutative differential calculi  into a sheaf $\dfs$ on $S^4$ by the assignment
\begin{equation}
\label{dfsheaf}
\dfs(U_N) := {\ndf} \, , \quad 
\dfs(U_S) := {\sdf} \, , \quad
\dfs(U_{SN}) := {\sndf}
\end{equation}
together with restriction maps
\begin{eqnarray}
\rho_{N,SN} & :   {\ndf}  \ra  {\sndf} \, , \qquad \d\b{ij}  & \mapsto \d\b{ij} \nn \\
\rho_{S,SN} & :   {\sdf}  \ra  {\sndf} \, , \qquad \d\a{ij} & \mapsto \widetilde{\Q}(\d\a{ij})=\d(\widetilde{\Q}(\a{ij})) \nn
\end{eqnarray}

\begin{defi} \label{page-Hodge}
The noncommutative differential calculus of the noncommutative 4-sphere $S^4_q$ is the sheaf of noncommutative algebras $\dfs$ over the classical 4-sphere $S^4$. 
\end{defi}

There is a Hodge duality operator on $\Omega^2_{S^4_q}$ which will be important when addressing  the self-duality of the instanton connection in \S \ref{sec:conn} . As discussed in more detail in \cite[\S 5.3]{clsII}, it coincides with the classical Hodge operator on $\Omega^{\bullet}_{S^4}$ since the toric symmetry beyond the 2-cocycle deformation induces a conformal transformation of the classical metric on $S^4$. 

We start on the local chart $U_N$, making use of the classical Hodge duality in $\mathbb{R}^4$. Denote by $x_i$ $(i=0,1,2,3)$ the coordinate functions on $\mathbb{R}^4$. The Hodge duality operator on two forms $\star:\Omega^2(\mathbb{R}^4)\rightarrow\Omega^2(\mathbb{R}^4)$ squares to the identity and provides a decomposition of $\Omega^2(\mathbb{R}^4)$ into self-dual (eigenvalue $+1$) and anti-selfdual (eigenvalue $-1$) differential forms:
$$\Omega^2(\mathbb{R}^4)= \Omega^{2,+}(\mathbb{R}^4)\oplus\Omega^{2,-}(\mathbb{R}^4) \, .$$
As $\A(\mathbb{R}^4)$-left modules the summands are generated by: 
\begin{align}
\Omega^{2,+}(\mathbb{R}^4) & =  \langle \, 
l_1^+=\d x_0\d x_1 + \d x_2\d x_3 \, , \, 
l_2^+=\d x_0\d x_2 - \d x_1\d x_3 \, , \, 
l_3^+=\d x_0\d x_3 + \d x_1\d x_2 \, \rangle \\
\Omega^{2,-}(\mathbb{R}^4) & =  \langle \, 
l_1^-=\d x_0\d x_1 - \d x_2\d x_3  \, , \, 
l_2^-=\d x_0\d x_2 + \d x_1\d x_3  \, , \, 
l_3^-=\d x_0\d x_3 - \d x_1\d x_2  \, \rangle \notag
\end{align}
To match our previous notation with $\rn$ and its generators, it is more convenient to consider complex coordinates
\begin{equation}
\label{betax} 
2\beta_{23} = x_0 + \i x_1 \, , \quad 
2\beta_{23}^{\ast} = x_0 - \i x_1 \, , \quad
2\beta_{24} = x_2 + \i x_3 \, , \quad
2\beta_{24}^{\ast} = x_2 - \i x_3 
\end{equation}
and, in view of our computations on the instanton curvature later on, to generate $\Omega^{2,-}(\mathbb{R}^4)$ with:
\begin{equation}
\label{ASD}
l_1^- = \d \beta_{23}^{\ast}\d \beta_{23} - \d \beta_{24}^{\ast}\beta_{24} \, , \quad  
l_2^- + l_3^- = \d\beta_{23}^{\ast}\d\beta_{24} \, , \quad
l_2^- - l_3^- = \d\beta_{24}^{\ast}\d\beta_{23}  \, .
\end{equation} 
When now the $\beta$ generators satisfy commutation relation \eqref{ncbeta} in $\rn$ (using \eqref{*beta} as well) the above expressions define anti-selfdual two-forms in $\ndf$. A completely analogous construction is performed in the chart $U_S$. By using the restriction maps of $\Omega^{\bullet}_q$ these local assignments define the sheaf (of $\shs$-modules) of anti-selfdual 2-forms $\Omega_q^{2,-}$.

\section{The quantum Hopf bundle as a sheaf of Hopf-Galois extensions}\label{sec-Hopf}

We now introduce a quantum principal bundle $\P$ - in the sense of \cite[\S 3]{pflaum} - with structure group $\aa = \aa(SU(2))$ on the quantum 4-sphere $S^4_q$.
We construct it out of a family of linear maps $\tau_{ij}$ (playing the role of `transition functions') from  the structure group $\aa $ to the double intersections of the two charts
$\shs(U_N)=\rn$ and $\shs(U_S)=\rs$. To facilitate the reading, we recall in a later Appendix (App.\ref{app:pflaum}) the formalism developed in  \cite{pflaum}  and the main results
relevant to our discussion.  Our construction is based on \cite[\S 3]{pflaum}, but we are here also concerned with the  condition of Hopf-Galois extension. This is an algebraic analogue of the geometric requirement that  the action in principal bundles is free and proper \cite{sch, bm}. In particular, in Prop. \ref{globalHG} below we prove  that the (global sections of the sheaf of the) total space  is an  $\aa(SU(2))$-Hopf-Galois extension of the quantum sphere $S^4_q$.

We start by recalling some results on Hopf comodule algebras and Hopf-Galois extensions. Let $H$ be a 
cosemisimple Hopf algebra with bijective antipode; if not otherwise stated, from now on we will always work with Hopf algebras of this kind. Let $P$ be a right $H$-comodule algebra, namely there exists a right $H$-coaction $\delta_H:P\rightarrow P\otimes H$, $\delta_H(p)=p_{(0)}\otimes p_{(1)}$, which in addition is an algebra morphism. The space of coinvariants $P^{co(H)}:=\{p\in P \mbox{ s.t. } \delta_H(p)=p\otimes 1\}$ is a subalgebra of $P$. The map $\can:P\otimes_{P^{co(H)}}P\rightarrow P\otimes H$
defined as $p'\otimes p \mapsto p'p_{(0)}\otimes p_{(1)}$ is usually referred to as the canonical map. When $\can$ is bijective we say that the extension $P^{co(H)}\subset P$ is an $H$-Hopf-Galois extension. 
An extension $B:=P^{co(H)}\subset P$ is cleft if and only if it is isomorphic to a crossed product $P\simeq B\#_{\sigma}H$. When the extention $B\subset P$ is Hopf-Galois, cleftness is equivalent to the normal basis property condition: $P\simeq B\tn H$ as left $B$-module and right $H$-comodule, where $B\tn H$ is a left $B$-module by left multiplication on the first factor and an $H$-comodule via $id \ot \Delta$ (see e.g.\cite[Thm. 7.2.2, Thm. 8.2.4]{mont}).
A morphism of $H$-comodule algebras $\phi:P\rightarrow P'$ is an algebra morphism which intertwines the $H$-coactions: $(\phi\otimes\id_H)\delta_H = \delta_{H}'\phi$. The following properties are easily verified: 
\begin{enumerate}[(a)]
\item
 $\phi$ maps coinvariants to coinvariants, $\phi(P^{co(H)})\subset P'^{co(H)}$; 
\item
 $\phi\otimes\phi$ is well-defined on $P\otimes_{P^{co(H)}}P$ and $(\phi\otimes\phi)(P\otimes_{P^{co(H)}}P)\subset P'\otimes_{P'^{co(H)}} P'$; 
\item
 $\phi$ intertwines the canonical maps, in the sense that $\can'(\phi\otimes\phi) = (\phi\otimes\id_H)\can$ on $P\otimes_{P^{co(H)}}P$, and also their lifts $\widetilde{\can}:P\otimes P\rightarrow P\otimes H$ and $\widetilde{\can'}:P'\otimes P'\rightarrow P'\otimes H$. 
\end{enumerate}
Finally, we recall that an extension $P^{co(H)}\subset P$ is Hopf-Galois if and only if it admits a strong connection (see e.g. \cite[\S 2.4]{hkmz} and reference therein). The latter is a unital linear map $l:H\rightarrow P\otimes P$, $l(h)=l(h)^{<1>}\otimes l(h)^{<2>}$, satisfying:
\begin{enumerate}[(i)]
\item $(l\otimes\id_H)\Delta_H = (\id_P\otimes\delta_H)\, l$
\item $(\id_H\otimes l)\Delta_H = (_H\delta\otimes\id_P)\, l$
\item $\widetilde{\can}\circ l = 1_P\otimes\id_H$
\end{enumerate}
where $_H\delta:P\rightarrow H\otimes P$, $_H\delta(p)= S^{-1}(p_{(1)})\otimes p_{(0)}$, is the left $H$-coaction induced from the right coaction $\delta_H$. Given a strong connection $l$, the inverse of the canonical map $\can$ is written as $ \can^{-1}(p\otimes h)= pl(h)^{<1>}\otimes_{P^{co(H)}} l(h)^{<2>}$.
\begin{lem}
\label{Hrestr}
Let $\phi:P\rightarrow P'$ be a morphism of unital right $H$-comodule algebras. If $P$ is an $H$-Hopf-Galois extension, so is $P'$.
\end{lem}
\begin{proof}
By hypothesis we have a strong connection $l:H\rightarrow P\otimes P$. We claim that $l':=(\phi\otimes\phi)\, l:H\rightarrow P'\otimes P'$ is a strong connection on $P'$. Property (i) on $l'$ amounts to the commutativity of the diagram
\begin{equation}
\xymatrix @C=4em{
H \ar[r]^l \ar[d]_{\Delta_H} & P\otimes P \ar[r]^{\phi\otimes\phi} \ar[d]^{\id_P\otimes\delta_H} & P'\otimes P' \ar[d]^{\id_{P'}\otimes\delta_{H}'} \\
H\otimes H \ar[r]_{l\otimes\id_H} & P\otimes P\otimes H \ar[r]_{\phi\otimes\phi\otimes\id_H} & P'\otimes P'\otimes H
}
\end{equation} 
which follows from the commutativity of the two sub-diagrams (due resp. to  property (i) on $l$ and to the fact that $\phi$ is an $H$-comodule algebra morphism). Similarly for property (ii). Finally property (iii) on $l'$ is displayed as the commutativity of the diagram
\begin{equation}
\xymatrix @C=4em{
H \ar[r]^l \ar@/_1pc/[dr]_{1_P\otimes\id_H} & P\otimes P \ar[d]^{\widetilde{\can}} \ar[r]^{\phi\otimes\phi} & P'\otimes P' \ar[d]^{\widetilde{\can}'}\\
& P\otimes H \ar[r]_{\phi\otimes\id_H} & P'\otimes H
}
\end{equation} 
which is again a consequence of the commutativity of the two sub-diagrams. 
\end{proof}

\begin{rem}
The geometric counterpart of the previous Lemma is the well known fact that given two $G$-spaces $X$ and $Y$ and a $G$-equivariant morphism $f:X\rightarrow Y$, if the $G$-action on $Y$ is free so is the one on $X$.
\end{rem}

\noindent
We introduce the following definition:
\begin{defi}
\label{HGsh}
Let $X$ be a topological space. Let $\mathcal{F}$ be a sheaf of (not necessarily commutative) algebras over  $X$ and $H$ a Hopf algebra. We say that $\mathcal{F}$ is a sheaf of $H$-Hopf-Galois extensions if: 
\begin{enumerate}[(i)]
\item $\mathcal{F}$ is a sheaf of (say) right $H$-comodules algebras and for each $W \subset U$ the restriction map $\rho_{UW}:\mathcal{F}(U) \ra \mathcal{F}(W)$ is a morphism of $H$-comodule algebras;
\item for each $U\subset X$ open set, $\mathcal{F}(U)^{co(H)}\subseteq \mathcal{F}(U)$ is a Hopf-Galois extension. 
\end{enumerate}
We denote by $\mathcal{F}^{co(H)}$ the sheaf on $X$ which associates to each open set $U$ the  subalgebra of coinvariants $\mathcal{F}(U)^{co(H)}$; we call it the subsheaf of coinvariants.
\end{defi}

Note that by  Lemma \ref{Hrestr} applied to the restriction maps, once $\mathcal{F}(U)$ is an $H$-Hopf-Galois extension, then $\mathcal{F}(W)$ is an $H$-Hopf-Galois extension for any open set $W\subseteq U$. In particular, if the algebra of global sections $\mathcal{F}(X)$ is an $H$-Hopf-Galois extension then  automatically $\mathcal{F}$ is a sheaf of $H$-Hopf-Galois extensions over $X$. We then see that the property of being a Hopf-Galois extension restricts locally. The converse `gluing property' is true for flabby sheaves (i.e. when restriction maps are surjective), as from the general theory of piecewise principality \cite[Thm. 3.3 and Corol. 3.10]{hkmz}. Namely, given a flabby sheaf $\mathcal{F}$ and an open set $U$ with a covering $\{U_i\}_{i\in I}$, if $\mathcal{F}(U_i)$ is an $H$-Hopf-Galois extension for any $i\in I$ then also $\mathcal{F}(U)$ is an $H$-Hopf-Galois extension. In fact in \cite{hkmz} the authors consider principal extensions, which for $H$ cosemisimple are equivalent to Hopf-Galois extensions, and they work with a family of $H$-comodule surjections $\pi_i:P\rightarrow P_i$, $i\in I$, such that $\cap_i (\mathrm{ker}\,\pi_i)=0$. In Theorem 3.3 they show that $P$ is principal if and only if the $P_i$'s are principal. To recover our setting it suffices to let $\pi_i$ be the restriction map $\mathcal{F}(X)\rightarrow\mathcal{F}(U_i)$, as they point out in Corollary 3.10. 

A natural class of examples of sheaves of Hopf-Galois extensions comes from smooth principal bundles.

\begin{ex}
\label{shF} (cf. \cite[Prop. 1.4.]{pflaum})
Let $M$ and $P$ be locally compact topological spaces, $G$ a matrix Lie group and $\pi:P\rightarrow M$ a principal $G$-bundle over $M$ which locally trivializes with respect to some covering $\{U_i\}_{i\in I}$ of $M$,  $\pi^{-1}(U_i)\simeq U_i\times G$. Let $H$ be the Hopf algebra of coordinate functions on $G$ and $\mathcal{M}$ the sheaf of functions on $M$. We can use the local trivialization of $P$ to define a sheaf $\mathcal{F}$ of $H$-Hopf-Galois extensions over $M$ as follows: for each $U_i$ set $\mathcal{F}(U_i)=\mathcal{M}(U_i)\otimes H$. The restriction maps are the restriction maps of $\mathcal{M}$ tensored with the identity on $H$, so they are surjective. Each $\mathcal{F}(U_i)$ is easily seen to be an $H$-Hopf-Galois extension 
(in particular cleft), and by the above discussed gluing property also $\mathcal{F}(M)$, which geometrically corresponds to the algebra of smooth functions on $P$, is an $H$-Hopf-Galois extension (albeit in general not a cleft one). 
\end{ex}

The previous example naturally suggest the following definition. 

\begin{defi}
A sheaf $\mathcal{F}$ of $H$-Hopf-Galois extensions over a topological space $X$ is called locally cleft if there exists an open covering $\{U_i\}_{i\in I}$ of $X$ such that $\mathcal{F}(U_i)$ is cleft, $\forall i\in I$.
\end{defi}

The notions of quantum principal bundle introduced by Pflaum and that of locally cleft sheaf of Hopf-Galois extensions are closely related.

\begin{prop}
\label{pfhg}
A sufficient condition for a quantum principal bundle $\mathcal{P}$ to be a sheaf of Hopf-Galois extensions (in fact, locally cleft) is that $\mathcal{P}$ is a flabby sheaf. In the opposite direction, every locally cleft sheaf of Hopf-Galois extensions is a quantum principal bundle.
\end{prop}
\begin{proof}
Let us consider a quantum principal bundle $\mathcal{P}$ (with base quantum space $\mathcal{X}$) on $X$. According to \cite[Def. 3.1]{pflaum}, there exists an open covering $\{U_i\}_{I\in I}$ of $X$ and a family of sheaf isomorphisms $\Omega_i:\mathcal{X}(U_i)\sharp_i H \rightarrow \mathcal{P}(U_i)$, hence $\mathcal{P}$ restricts locally to cleft Hopf-Galois extensions. As we have already discussed above, a sufficient condition for these to glue to a well-defined Hopf-Galois sheaf  is that the restriction maps are surjective \cite[Corol. 3.10]{hkmz}.

To prove the second statement, suppose that $\mathcal{P}$ is a locally cleft sheaf of Hopf-Galois extensions over $X$. Then we have an injective sheaf morphism $\rho:\mathcal{P}^{co(H)}\hookrightarrow\mathcal{P}$ and an open covering $\{U_i\}_{i\in I}$ of $X$ such that $\mathcal{P}(U_i)$ is cleft. The $H$-comodule sheaf morphisms $\Omega_i(U):\mathcal{P}^{co(H)}(U)\sharp_i H\rightarrow \mathcal{P}(U)$ defined as $f\otimes h\mapsto f\gamma_i(h)$, where $U\subset U_i$ and $\gamma_i:H\rightarrow\mathcal{P}(U_i)$ is the cleaving map, are in fact isomorphism \cite[Thm. 7.2.2]{mont}. It is now trivial to check that the datum of $(\mathcal{P},\mathcal{P}^{co(H)},\rho,H,(\Omega_i)_{i\in I})$ corresponds to Pflaum's definition of a quantum principal bundle.
\end{proof}
\medskip 
In the remaining of the section we construct a quantum principal bundle $\mathcal{P}$ over the noncommutative 4-sphere $S^4_q$. It mimics the sheaf of the principal $SU(2)$-Hopf bundle, therefore it will be referred to as the quantum Hopf bundle over $S^4_q$. We will see that it is not flabby, nevertheless we prove it to be a sheaf of locally cleft Hopf-Galois extensions. This is an example of how the flabbiness hypothesis in Proposition \ref{pfhg} is a sufficient but not necessary condition. 

We consider  the covering of $S^4$ consisting of the two open sets $U_N,~U_S$ as before and the sheaf $\shs$ introduced in \eqref{sheafB} above. By using Prop. \ref{prop:fact}  (and understanding the isomorphism $\ai \simeq \shs(U_N \cap U_S)$) we introduce 
`transition functions' $\tau_{ij}, ~i,j\in \{N,S\}$  as the linear maps
\begin{align}
\label{trans}
\tnn: \aa &\ra \ai \, , & 
\tss: \aa &\ra \ai \, , &  
\tns: \aa &\ra \ai \, , & 
\tsn: \aa &\ra \ai \\   
h &\mapsto \varepsilon (h)1\otimes 1 \, , &
h &\mapsto \varepsilon (h)1\otimes 1 \, , &
h &\mapsto h \ot 1 \, , & 
h &\mapsto S(h) \ot 1 \nn
\end{align}
for each $h \in \aa$.  It is promptly proved that the maps $\tau_{ij}$ form an $\aa$-cocycle in $\shs$ in the sense of \cite[Def. 3.11]{pflaum}. 
Moreover the  $\tau_{ij},~i,j\in \{N,S\}$  above are algebra morphisms ($\tsn$ as well, despite the presence of the antipode, since $\aa$ is commutative). 
\smallskip

The general theory, see the Appendix \ref{app:pflaum},  gives a recipe for constructing a quantum principal bundle out of a set of transition functions. The sheaf $\P$ of right $\aa$-comodule algebras is defined by setting  

\begin{eqnarray}\label{sheafPN}
\mathcal{P}(U_N)&:=& \left\{(b^N , b^{SN}) \in \left(\shs(U_N) \ot \aa \right) \oplus  \left(\shs(U_S \cap U_N) \ot \aa \right) \right. \nn \\
&& \left. s.t.~ (\rho_{N,SN} \ot id) (b^N)=(m \ot id)(id \ot f_{SN}^{-1} \circ \tau_{NS} \ot id)(id \ot \Delta) (b^{SN}) \right\} \label{condPun}
\end{eqnarray}
and similarly
\begin{eqnarray}\label{sheafPS}
\mathcal{P}(U_S)&:=& \left\{(b^S , b^{SN}) \in \left(\shs(U_S) \ot \aa \right) \oplus  \left(\shs(U_S \cap U_N) \ot \aa \right) \right. \nn \\
&& \left. s.t. ~(\rho_{S,SN} \ot id) (b^S)=(m \ot id)(id \ot f_{SN}^{-1} \circ\tau_{SN} \ot id)(id \ot \Delta) (b^{SN}) \right\}
\end{eqnarray}
while on the intersection $\mathcal{P}$ is simply given by
\be
\mathcal{P}(U_S \cap U_N):= \shs(U_S \cap U_N) \ot \aa \, .
\ee
Finally
\begin{eqnarray}\label{sheafPSN}
\hspace*{-1cm}
\mathcal{P}(S^4)&:=& \left\{(b^N , b^{S}) \in \left(\shs( U_N) \ot \aa \right) \oplus  \left(\shs(U_S) \ot \aa \right)    \right. \nn \\
&& \left. s.t. ~(\rho_{N,SN} \ot id) (b^N)=(m \ot id)(id \ot f_{SN}^{-1} \circ\tau_{NS} \ot id)(id \ot \Delta) (\rho_{S,SN} \ot id)(b^{S}) \right\} \, .
\end{eqnarray}
\noindent
We note that we can constructively find pairs $(b^N,b^{SN})$ belonging to $\mathcal{P}(U_N)$ by defining (see \cite[Lemma 3.13]{pflaum})
$$b^{SN} := (m\tn id)(id\tn f_{SN}^{-1}\circ\tau_{NS}\tn id)(id\tn\Delta)(\rho_{N,SN}\tn id)b^N \, .$$
Hence one defines the trivialization morphism 
\begin{equation}
\label{trivN}
\begin{array}{ccl}
\Omega_N: \shs(U_N)\tn \aa & \rightarrow & \mathcal{P}(U_N) \\
\quad\quad b^N & \mapsto & (b^N,(m\tn id)(id\tn f_{SN}^{-1}\circ\tau_{NS}\tn id)(id\tn\Delta)(\rho_{N,SN}\tn id)\, b^N) \; ,
\end{array}
\end{equation}
which is an  isomorphism of right $\aa$-comodule algebras. The same construction applies to pairs in $\mathcal{P}(U_S)$ and to the trivialization morphism  $\Omega_S: \shs(U_S)\tn \aa \rightarrow \mathcal{P}(U_S)$. 
\smallskip

We already pointed out that the restriction maps $\rho_{N,SN}$ and $\rho_{S,SN}$ of the sheaf $\shs$ are not surjective. As a consequence, the sheaf $\mathcal{P}$ of the quantum Hopf bundle  is not flabby. 
We prove that nevertheless it is a sheaf of Hopf-Galois extensions, locally cleft  due to the trivialization isomorphisms \eqref{trivN}. In view of Lemma \ref{Hrestr} (also see the discussion after Definition \ref{HGsh}) it suffices to show that the global sections are a Hopf-Galois extension.

\begin{prop}
\label{globalHG}
The subalgebra of coinvariants  $
\co:= (\mathcal{P}(S^4))^{co(SU(2))}$  is
\begin{equation}\label{Bcoin}
\co = \{ (x \ot 1, y \ot 1) \in \P(S^4)  ~\slash ~ \rho_{N,SN} (x) = \rho_{S,SN} (y) \} = \shs(S^4)  \, .
\end{equation}
The extension $\co \subset \P(S^4)$ is Hopf-Galois. Furthermore $\P(S^4) $ is a faithfully flat $\co$-module. 
\end{prop}
\begin{proof}
The coaction on $\P(S^4)$ is the restriction of the direct sum coaction on $\left(\shs( U_N) \ot \aa \right) \oplus  \left(\shs(U_S) \ot \aa \right)$ given on each summand by the right regular corepresentation $(id \ot \Delta)$. From this one gets the
 explicit form \eqref{Bcoin} of $\co$. To show that $\co \subset \P(S^4)$ is Hopf-Galois we exhibit a strong connection $l:\aa\rightarrow\P(S^4)\otimes\P(S^4)$; we set
\begin{equation}
l(h) := (1\otimes S(h_{(1)}),0)\otimes (1\otimes h_{(2)}) + (0,1\otimes S(h_{(1)}))\otimes (0,1\otimes h_{(2)}) \, . 
\end{equation}
We check the three properties $l$ has to satisfy (see (i), (ii) and (iii) before Lemma \ref{Hrestr}). We begin with $(l\otimes\id)\Delta = (\id\otimes\delta_\aa)\, l$. The left hand side reads
\begin{equation*}
\begin{split}
(l\otimes\id)\Delta(h) & = l(h_{(1)})\otimes l(h_{(2)}) \\
& = \big( (1\otimes S(h_{(1)}),0)\otimes (1\otimes h_{(2)}) + (0,1\otimes S(h_{(1)}))\otimes (0,1\otimes h_{(2)}) \big) \otimes h_{(3)}
\end{split}
\end{equation*}
which agrees with the right hand side
\begin{equation*}
(\id\otimes\delta_\aa)\, l(h)= (\id\otimes\delta_\aa) \, \big( (1\otimes S(h_{(1)}),0)\otimes (1\otimes h_{(2)}) + (0,1\otimes S(h_{(1)}))\otimes (0,1\otimes h_{(2)}) \big)
\end{equation*}
once the explicit form of the coaction  is taken into account. Similarly for 
$(\id\otimes l)\Delta = (_\aa\delta\otimes\id)\, l$, where the left hand side is computed as
\begin{equation*}
\begin{split}
(\id\otimes l)\Delta(h) & = h_{(1)}\otimes l(h_{(2)}) \\
& = h_{(1)} \otimes \, \big(
(1\otimes S(h_{(2)}),0)\otimes (1\otimes h_{(3)}) + (0,1\otimes S(h_{(2)}))\otimes (0,1\otimes h_{(3)}) \big)
\end{split}
\end{equation*}
and the right hand side as
\begin{equation*}
\begin{split}
(_H\delta & \otimes\id)\, l(h) = (_\aa\delta\otimes\id) \, \big( (1\otimes S(h_{(1)}),0)\otimes (1\otimes h_{(2)}) + (0,1\otimes S(h_{(1)}))\otimes (0,1\otimes h_{(2)}) \big) \\
& = S^{-1}(S(h_{(1)}))_{(2)} \otimes \big( (1\otimes (S(h_{(1)}))_{(1)},0)\otimes (1\otimes h_{(2)},0) + (0,1\otimes (S(h_{(1)}))_{(1)})\otimes (0,1\otimes h_{(2)}) \big) \\
& = S^{-1}(S(h_{(1)})) \otimes \big( (1\otimes S(h_{(2)}),0)\otimes (1\otimes h_{(3)},0) +
(0,1\otimes S(h_{(2)})\otimes (0,1\otimes h_{(3)}) \big) \, .
\end{split}
\end{equation*}
Finally, $\widetilde{\can}\circ l = 1\otimes\id$. Indeed
\begin{equation*}
\begin{split}
\widetilde{\can} \, ( l(h) ) & = \widetilde{\can} \, \big( (1\otimes S(h_{(1)})),0)\otimes (1\otimes h_{(2)}) + (0,1\otimes S(h_{(1)}))\otimes (0,1\otimes h_{(2)}) \big) \\
& = \big( (1\otimes S(h_{(1)})),0) \cdot (1\otimes h_{(2)}) + 
(0,1\otimes S(h_{(1)})) \cdot (0,1\otimes h_{(2)}) \big) \otimes h_{(3)} \\
& = \big( (1\otimes 1,0) + (0,1\otimes 1) \big) \otimes \varepsilon(h_{(1)})\, h_{(2)} = (1\otimes 1,1\otimes 1)\otimes h \, .
\end{split}
\end{equation*}
The last assertion  follows from the property of  $\aa(SU(2)) $ to be cosemisimple.
\end{proof}

\subsection{The instanton sheaf}\label{sec:instSheaf}

Let us consider the fundamental left corepresentation of $SU(2)$, $\rho: \IC^2 \ra \aa \ot \IC^2$, $(z_1,z_2) \mapsto A \pot (z_1,z_2)$, where $A$ is the defining matrix introduced in \eqref{matrixA} and $\pot$ indicates tensor product and matrix multiplication combination.   We will use the notation $\rho(z)= \muno{z} \ot \zero{z}$ for the left coaction $\rho$ on $z=(z_1,z_2) \in \IC^2$.
For each open set $U$ of $S^4$ on which the sheaf $\P$ of the quantum Hopf bundle trivializes, the algebra $\IC^2 \ot \P(U)$ can be endowed with a right $\aa$-comodule 
structure via
\be
\Psi: \IC^2 \ot \P(U) \ra \IC^2 \ot \P(U) \ot \aa, \quad z \ot x \mapsto \zero{z} \ot \zero{x} \ot S^{-1}(\muno{z}) \uno{x} \, .
\ee
Set $\V(U):= \left(\IC^2 \ot \P(U) \right)^{co (\aa)}= \left\{ z \ot x ~| ~\Psi(z \ot x)= z \ot x \ot 1\right\} \subseteq \IC^2 \ot \P(U)$ and endow it with the algebra structure inherited from $\IC^2 \ot \P(U)$.
The assignment $U \mapsto \V(U)$   defines a sheaf of algebras $\V$ on $S^4$ (with restriction maps given by extending those of $\P$ to the tensor product via the identity on $\IC^2$). In agreement with \cite[\S 4.2]{pflaum} we refer to $\V$ as the associated quantum vector bundle to $\P$ with typical fiber $\IC^2$. The following result provides the usual equivalent characterization of the associated bundle in terms of a base space  and a typical fiber.

\begin{prop}
For $U$ as above, there exists an isomorphism of algebras
$\V(U) \simeq \shs(U)\ot \IC^2$.
\end{prop}
\begin{proof}
For $U=U_N$ consider
the map
\begin{eqnarray}
\Gamma_N: \shs(U_N)\ot \IC^2  & \ra & \V(U_N)=(\IC^2\tn\mathcal{P}(U_N))^{co(\aa)} \nn \\
a_N\tn z & \mapsto & z_{(0)}\tn\Omega_N(a_N\tn z_{(-1)})  \, .
\end{eqnarray}
Since $\Omega_N$ is a morphism of right $\aa$-comodules (see \eqref{trivN}), $\Gamma_N$ takes indeed values in the $\aa$-coinvariant subalgebra of $\IC^2\tn\mathcal{P}(U_N)$. The inverse of $\Gamma_N$ on a generic $(z\tn(b^N\oplus b^{SN}))\in (\IC^2\tn\mathcal{P}(U_N))^{co(\aa)}$ is proven to be $p_1(\Omega_N^{-1}(b^N\oplus b^{SN}))\tn z$, where $p_1$ is the projection onto the first factor. For $U=U_S$ the construction is similar.  	
\end{proof}

The quantum principal bundle and quantum associated bundle discussed so far reduce to the classical Hopf bundle and associated instanton bundle with topological charge (or instanton number) equal to $1$. We recall that classically the instanton number can be characterized as the degree $k\in\pi_3(S^3)=\mathbb{Z}$ of the transition map $\tau_{NS}:SU(2)\simeq S^3\ra U_N\cap U_S\simeq S^3\times I\sim S^3$, where $\sim$ (resp. $\simeq$) stands for homotopic (resp. topological) equivalence.
Quantum bundles with higher instanton numbers are obtained via transition functions of higher (topological) degree; for $k\in\mathbb{N}$ we set $\tnn^k:=\tnn$  and $\tss^k:=\tss$ as in \eqref{trans}, while
\begin{align}
\tns^k: \aa &\ra \ai \, , & 
\tsn^k: \aa &\ra \ai \nn \\   
h &\mapsto h_{(1)}\cdot\ldots\cdot h_{(k)} \ot 1 \, , & 
h &\mapsto S(h_{(k)})\cdot\ldots\cdot S(h_{(1)}) \ot 1 \nn \, .
\end{align}
The resulting principal and associated bundles are referred to as the quantum $SU(2)$-instanton bundles on $S^4_q$  with charge $k\in\mathbb{N}$. One gets negative charges by exchanging $\tns^k$ with $\tsn^k$. For $q=1$ they reduce to the classical $SU(2)$-instanton bundles on $S^4$ of  corresponding charge.

\section{The connection and its anti-self dual curvature}\label{sec:conn}

We define an anti-selfdual connection on the $SU(2)$-Hopf bundle on $S^4_q$. As in many other constructions in this paper, the advantage of a sheaf theoretic approach is that we can work locally on the two patches $U_N$ and $U_S$ in order to describe global objects on $S^4_q$. The local data is assembled together by using the restriction maps written in terms of the isomorphism $\Q$, as done for example in \eqref{glsshs}. We present the explicit formulas for the connection and the curvature on $U_N$ only, the case on $U_S$ being completely similar.
  
Let us consider the following two one-forms in $\ndf$: 
\begin{eqnarray}
\eta_1 &:=& \b{23}^* \d\b{23} + q^2 \b{24} \d\b{24}^* - \d\b{23}^* \b{23} - q^2 \d\b{24} \b{24}^*
\\
\eta_2 &:=&  2(\b{23}^* \d\b{24} - q^2 \b{24} \d\b{23}^*) \, .
\end{eqnarray}
\noindent
We derive some identities to be used shortly after. For pure computational reasons, in order to compare and simplify  the monomials appearing in the expressions below, we choose the following (arbitrary) order among zero and one forms:
$$
\b{23}^* < \b{23} <\b{24} < \b{24}^* <\d\b{23}^* < \d\b{23} <\d\b{24} < \d\b{24}^* \, .
$$

\begin{lem}
$
\eta_1 \wedge_q \eta_1 =0$.
\end{lem}
\begin{proof} 
By using the commutation relations in $\ndf$ and omitting the $\wedge_q$ mark we have
\begin{eqnarray*}
\eta_1 \wedge_q \eta_1 &=& 
q^2 \b{23}^* \d\b{23} \b{24} \d\b{24}^*
- \b{23}^* \d\b{23} \d\b{23}^* \b{23}
- q^2 \b{23}^* \d\b{23} \d\b{24} \b{24}^*
+ q^2 \b{24} \d\b{24}^* \b{23}^* \d\b{23} + \\
&&
- q^2 \b{24} \d\b{24}^* \d\b{23}^* \b{23}
- q^4\b{24} \d\b{24}^* \d\b{24} \b{24}^*
- \d \b{23}^* \b{23} \b{23}^* \d\b{23}
- q^2 \d\b{23}^* \b{23} \b{24} \d\b{24}^*+
\\ &&
q^2  \d\b{23}^* \b{23}  \d\b{24} \b{24}^*  
- q^2 \d\b{24} \b{24}^* \b{23}^* \d\b{23} 
- q^4 \d\b{24}   \b{24}^* \b{24}  \d\b{24}^*
+ q^2  \d\b{24}    \b{24}^* \d\b{23}^*  \b{23}
\\  [3mm]
&=& 
q^2 \b{23}^*  \b{24} \d\b{23} \d\b{24}^*
+ \b{23}^* \b{23} \d\b{23}^* \d\b{23} 
- q^{-2} \b{23}^*  \b{24}^*\d\b{23} \d\b{24}
- q^2 \b{23}^* \b{24}   \d\b{23} \d\b{24}^* + \\
&&
q^6 \b{23} \b{24}  \d\b{23}^* \d\b{24}^* 
+ q^4\b{24}  \b{24}^*  \d\b{24} \d\b{24}^* 
-  \b{23}^*\b{23}  \d \b{23}^*\d\b{23}
- q^6  \b{23} \b{24} \d\b{23}^* \d\b{24}^*+
\\ &&
q^2  \b{23} \b{24}^*  \d\b{23}^* \d\b{24}
+ q^{-2} \b{23}^* \b{24}^*\d\b{23}\d\b{24}
- q^4  \b{24}  \b{24}^* \d\b{24}\d\b{24}^*
- q^2  \b{23} \b{24}^* \d\b{23}^* \d\b{24}
\\  [3mm]
&=& 0 \; .
\end{eqnarray*}
\end{proof}
\begin{lem}\label{lemma12}
$
\eta_1 \wedge_q \eta_2 = - \eta_2 \wedge_q \eta_1$.
\end{lem}
\begin{proof}
By a direct check. On the one hand
\begin{eqnarray*}
\frac{1}{2} \eta_1 \wedge_q \eta_2 &=& 
\b{23}^* \d\b{23} \b{23}^* \d\b{24}
-q^2 \b{23}^* \d\b{23} \b{24} \d\b{23}^*
+ q^2 \b{24} \d\b{24}^* \b{23}^* \d\b{24}
- q^4 \b{24} \d\b{24}^* \b{24} \d\b{23}^* 
\\ &&
-q^2 \b{23}\d\b{23}^* \b{23}^* \d\b{24}
+ q^2 \b{24}^* \d\b{24} \b{24} \d\b{23}^*
\\  [3mm]
&=& 
q^{-2}\b{23}^* \b{23}^* \d\b{23}  \d\b{24}
+ \b{23}^*  \b{24} \d\b{23}^* \d\b{23}
- q^2 \b{23}^* \b{24}  \d\b{24}\d\b{24}^*  
+ q^6 \b{24}  \b{24} \d\b{23}^* \d\b{24}^*
\\ &&
-\b{23}^* \b{23}  \d\b{23}^* \d\b{24}
- q^2  \b{24} \b{24}^* \d\b{23}^* \d\b{24}  \; .
\end{eqnarray*}
On the other hand
\begin{eqnarray*}
\frac{1}{2} \eta_2 \wedge_q \eta_1 &=& 
\b{23}^* \d\b{24} \b{23}^* \d\b{23}
+q^2 \b{23}^* \d\b{24} \b{24} \d\b{24}^*
- q^2 \b{23}^* \d\b{24} \b{23} \d\b{23}^*
- q^2 \b{24} \d\b{23}^* \b{23}^* \d\b{23} 
\\ &&
-q^4 \b{24}\d\b{23}^* \b{24} \d\b{24}^*
+ q^2 \b{24} \d\b{23}^* \b{24}^* \d\b{24}
\\  [3mm]
&=& 
- q^{-2}\b{23}^* \b{23}^* \d\b{23}  \d\b{24}
+ q^2 \b{23}^* \b{24}  \d\b{24}\d\b{24}^*  
+\b{23}^* \b{23}  \d\b{23}^* \d\b{24}
- \b{23}^*  \b{24} \d\b{23}^* \d\b{23}
\\ &&
- q^6 \b{24}  \b{24} \d\b{23}^* \d\b{24}^*
+ q^2  \b{24} \b{24}^* \d\b{23}^* \d\b{24} 
\end{eqnarray*}
so that $\eta_1 \wedge_q \eta_2 = - \eta_2 \wedge_q \eta_1$.
\end{proof} 

\begin{lem}\label{lemma22*}
$\eta_2 \wedge_q \eta_2^* = - \eta_2^* \wedge_q \eta_2$.
\end{lem}
\begin{proof}
Similarly to above, with some algebra we compute
\begin{eqnarray*}
\frac{1}{4}\eta_2 \wedge_q \eta_2^* &=&
q^2 \b{23}^* \b{23} \d\b{24}  \d\b{24}^*
+ q^{-2} \b{23}^* \b{24}^*  \d\b{23}\d\b{24}  
-q^6\b{23} \b{24}  \d\b{23}^* \d\b{24}^*
+q^2 \b{24}  \b{24}^* \d\b{23}^* \d\b{23}  \\&=& - \, \frac{1}{4}\eta_2^* \wedge_q \eta_2 \; .
\end{eqnarray*}
\end{proof} 
~ \smallskip
Observe that in $\ndf$ we have 
$$
\eta_1 r^2= r^2 \eta_1 \; , \quad \eta_2 r^2= r^2 \eta_2 
$$
for $r^2= \b{24}^*\b{24} + \b{23}^*\b{23} $ (cf. \eqref{r2}). Let us extend the algebra $\ndf$ by a generator $t$ and its differential $d(t)=dt$ and quotient by the relation 
$$
t(1+r^2)=1= (1+r^2)t \; .
$$
By imposing the Leibniz rule, from the previous equation we get $dt=-t^2 d(r^2)$. Furthermore,
since $r^2 \eta_i = \eta_i r^2$ it follows that $t$ (and $\d t$) has to commute with the $\eta_i$, $i=1,2$. 
\\
\noindent
Consider now the matrix 
\be\label{con-mat}
\con:= \frac{1}{2} t \begin{pmatrix}
\eta_1 & \eta_2
\\
- \eta_2^* & \eta_1^*
\end{pmatrix}
\ee
We observe that $\eta_1^*=-\eta_1$ and $\con \in su(2) \ot {_N}\Omega^1_q$ , so that $\con$ is the local restriction to $U_N$ of a connection one-form on the quantum Hopf  bundle. 

By using the Lemmas above, the curvature $F_\con=\d \con + \con \wedge_q \con$ of the $SU(2)$ potential $\A$ reduces to 
\be\label{cur}
F_\con=
\frac{1}{2}  \begin{pmatrix}
- t^2 dr^2 \wedge_q \eta_1 + t\d\eta_1 - \frac{1}{2} t^2   \eta_2 \wedge_q \eta_2^* ~~&
-t^2 dr^2 \wedge_q \eta_2 + t \d\eta_2 + t^2 \eta_1\wedge_q \eta_2
\\
\\
t^2 dr^2 \wedge_q \eta_2^* -t\d \eta_2^* - t^2 \eta_2^* \wedge_q \eta_1& 
t^2 \d r^2 \wedge_q \eta_1 -t \d \eta_1 + \frac{1}{2} t^2 \eta_2 \wedge_q \eta_2^*
\end{pmatrix}
\ee
and $F_\con$ is an $su(2)$-valued two-form on $\rn$ as expected. 
\begin{thm}
The curvature $F_\con$ has the expression
\be\label{curv}
F_\con= t^2  \begin{pmatrix} 
\d\b{23}^* \d\b{23} 
+ q^2  \d\b{24}  \d\b{24}^* 
& 2  \d\b{23}^* \d\b{24}
\\
\\
- 2  \d\b{24}^* \d\b{23} & - \d\b{23}^* \d\b{23} 
- q^2  \d\b{24}  \d\b{24}^* 
\end{pmatrix}
\ee
and it is anti-selfdual, $\star_q F_\con= - \, F_\con$.
\end{thm}
\begin{proof}
We start by computing the single summand of the entrance  $(F_\con)_{11}$: 
\begin{eqnarray*}
- \frac{1}{2} t^2 dr^2 \wedge_q \eta_1 &=&  -\frac{1}{2}
t^2 \left(
\b{23}^*\b{23} \d \b{23}^* \d\b{23} 
+ q^6 \b{23}\b{24} \d \b{23}^* \d\b{24}^*
-q^2 \b{23}\b{24}^* \d \b{23}^* \d\b{24} 
+q^2 \b{23}^*\b{24} \d \b{23} \d\b{24}^* 
\right.
\\
&&
+ \b{23}^*\b{23} \d \b{23}^* \d\b{23} 
-q^{-2}\b{23}^*\b{24}^* \d \b{23} \d\b{24} 
-q^2\b{23}^*\b{24} \d \b{23} \d\b{24}^* 
+ q^6 \b{23}\b{24} \d \b{23}^* \d\b{24}^* 
\\
&& \left.
+ q^4 \b{24}\b{24}^* \d \b{24} \d\b{24}^* 
-q^{-2} \b{23}^*\b{24}^* \d \b{23} \d\b{24} 
+ q^4 \b{24}\b{24}^* \d \b{24} \d\b{24}^* 
+q^2 \b{23}\b{24}^* \d \b{23}^* \d\b{24} 
\right)
\\
&=&  -
t^2 \left(
 \b{23}^*\b{23} \d \b{23}^* \d\b{23} 
+  q^6 \b{23}\b{24} \d \b{23}^* \d\b{24}^*
- q^{-2}\b{23}^*\b{24}^* \d \b{23} \d\b{24}  
+  q^4 \b{24}\b{24}^* \d \b{24} \d\b{24}^* 
\right) .
\end{eqnarray*}
The term involving the differential of $\eta_1$ is
$\frac{1}{2} t d\eta_1 =  t(\d \b{23}^* \d\b{23} + q^2  \d\b{24} \d \b{24}^* )$, while $\eta_2 \wedge_q \eta_2^*$ was already computed in Lemma \ref{lemma22*} and gives
$$ 
- \frac{1}{4} t^2 \eta_2 \wedge_q \eta_2^* = -t^2 \left(
q^2 \b{23}^* \b{23} \d\b{24}  \d\b{24}^*
+ q^{-2} \b{23}^* \b{24}^*  \d\b{23}\d\b{24}  
-q^6\b{23} \b{24}  \d\b{23}^* \d\b{24}^*
+q^2 \b{24}  \b{24}^* \d\b{23}^* \d\b{23}  \right)
.$$
Summing them we obtain
\begin{eqnarray*}
(F_\con)_{11} &=&  \left( -t^2 \b{23}^* \b{23} + t -q^2 t^2 \b{24} \b{24}^* \right) \left( \d\b{23}^* \d\b{23} 
+ q^2  \d\b{24}  \d\b{24}^* \right)
\\
&=&  \left( -t^2 r^2 + t \right) \left( \d\b{23}^* \d\b{23} 
+ q^2  \d\b{24}  \d\b{24}^* \right) 
\\
&=& t^2 \left( \d\b{23}^* \d\b{23} 
+ q^2  \d\b{24}  \d\b{24}^* \right).
\end{eqnarray*}
Similarly, we now compute  $(F_\con)_{12}$:
\begin{eqnarray*}
- \frac{1}{2} t^2 dr^2 \wedge_q \eta_2 &=&  -
t^2 \left( 
\b{23}^*\b{23} \d \b{23}^* \d\b{24} 
+ q^{-2} \b{23}^*\b{23}^* \d \b{23} \d\b{24}
+ \b{23}^*\b{24} \d \b{23}^* \d\b{23} 
-q^2 \b{23}^*\b{24} \d \b{24} \d\b{24}^* +
\right.
\\
&& \left.
+ q^6 \b{24}\b{24} \d \b{23}^* \d\b{24}^* 
+q^2 \b{24}\b{24}^* \d \b{23}^* \d\b{24} 
\right)\\
\frac{1}{2} t d\eta_2 &=& t(\d \b{23}^* \d\b{24} - q^2  \d\b{24} \d \b{23}^* )= 2 t \d \b{23}^* \d\b{24} \, .
\end{eqnarray*}
We sum the previous terms to $\frac{1}{2} t^2 \eta_1\wedge_q \eta_2$, whose expression was computed in \ref{lemma12};  the only terms which do not cancel give
$$
(F_\con)_{12} = \left( -2 t^2 \b{23}^* \b{23} +2 t -2 q^2 t^2 \b{24} \b{24}^* \right)  \d\b{23}^* \d\b{24} = \left( -2 t^2 r^2 + 2 t \right)  \d\b{23}^* \d\b{24} = 2 t^2 \d\b{23}^* \d\b{24} \, .
$$
Comparing this with the explicit expression of anti-selfdual forms in $\ndf$ given in \eqref{ASD}, we conclude that
$F_\con$ is anti-selfdual.
\end{proof}
\medskip

The above Theorem shows that  the connection $\con $ in  \eqref{con-mat}(together with its analogue in $\sdf$) describes an anti-instanton on $\sph$. For $q=1$ the potential $\con $ reduces to the basic anti-instanton of charge $k=-1$ and its curvature $F_\con$ agrees with $F_\con = (1+|\mathbf{x}|^2)^{-2} d\bar{\mathbf{x}}d\mathbf{x}$ in  quaternionic notation $\mathbf{x}= x_0+\i \,x_1 + \mathrm{j}\,x_2+ \mathrm{k}\,x_3$ (see \cite[Ch.II]{at}). For the relation among generators $\beta$  and $x_i$ see \eqref{betax}.

\appendix
\section{Quantum principal bundles}\label{app:pflaum}
In this appendix we recall the main results from the theory  of quantum principal bundles developed by Pflaum in \cite{pflaum}  which are used in our construction.\\

Let $M$ be a topological space with a covering $\mathcal{U}=(U_i)_{i\in I}$ and $\ma$ a subcategory of the category of (not necessarily commutative) associative algebras.  Let:
\begin{itemize}
\item[-] $\Mm$ and $\P$ be two sheaves on $M$ with objects in $\ma$  such that $\Mm(U), \P(U)$ are unitary for each open set $U \subset M$;  
\item[-]  $H$ be a Hopf algebra (with coproduct and counit denoted respectively by $\Delta$  and
$\varepsilon$);
\item[-] a sheaf morphism $\sigma: \Mm \ra \P$ such that the sequence $0 \ra \Mm \stackrel{\sigma}{\ra} \P$ is exact;
\item[-] a family of sheaf morphisms $(\Omega_i)_{i\in I}$, with $\Omega_i: \Mm_{|_{U_i}} \ot H \ra \P_{|_{U_i}}$ such that the sheaf morphisms
$\Omega_{ij} : \Mm_{|_{U_i \cap U_k}} \ot H \ra \Mm_{|_{U_i \cap U_k}} \ot H$ 
defined by 
\be
(\Omega_{ij})_U:=(\Omega_{i})^{-1}_U \circ (\Omega_{j})_U \; , \quad U \subseteq U_i \cap U_j \mbox{ open}
\ee
satisfy the following equations:
\begin{eqnarray*}
&& (\Omega_i)_U (f \ot 1)= \sigma_U(f) , \quad U \subseteq U_i, \; f \in \Mm(U)
\\
&& \left( (\Omega_{ij})_U \ot id \right) \circ(id \ot \Delta)= (id \ot \Delta)\circ (\Omega_{ij})_U ,
\quad U \subseteq U_i \cap U_j.
\end{eqnarray*}
\end{itemize}

The data $\left( \P,\Mm,\sigma,H,(\Omega_i)_{i\in I} \right)$ defines an $\ma$-quantum principal bundle over $M$ with total quantum space $\P$, base quantum space $\Mm$ and structure group $H$. 
\footnote{In \cite{pflaum} the author adopts a higher degree of generality and assumes  the local trivializations $\Omega_i$ to be  defined on  crossed products  
$\Mm(U_i) \#_i H$ of $\Mm(U_i)$ and $H$ arising from 2-cocycle deformations. Nevertheless for our scope we can simplify the theory and work with $\Mm(U_i) \ot H$.}
The principal (co)action is given by the sheaf morphism 
$
\phi: \P \ra \P \ot H$ uniquely determined (cf. \cite[Thm.3.7]{pflaum}) by the condition
\be\label{prin-coa}
\phi_U:= (\Omega_i \ot id)\circ (id \ot \Delta) \circ \Omega_i^{-1}, \quad U \subseteq U_i \;.
\ee
For each open set $U$, $\phi_U$ defines a right $H$-coaction turning $\P(U)$ into a right $H$-comodule.
The `local coordinate changes' $\Omega_{ij}$ allows for the definition of a family of linear maps 
\be
\tau_{ij}:H \ra \Mm(U_i \cap U_j) , \quad \tau_{ij}(h):= (id \ot \varepsilon)\circ \Omega_{ij}(1 \ot h),\quad i,j \in I, ~h \in H
\ee
which by construction  satisfy the following conditions:
\begin{itemize}
\item[(i)] $\tau_{i,i}(1)=1$ ~;
\item[(ii)] $\tau_{ii}(h)=\varepsilon(h)\cdot 1 , ~h \in H$ ~;
\item[(iii)] $\rho_{U_i \cap U_j,U} \circ \tau_{ij}= (\rho_{U_i \cap U_k,U} \circ \tau_{ik}) * (\rho_{U_k \cap U_j,U} \circ \tau_{kj}) $, for $U \subseteq U_i \cap U_j \cap U_k$,
\end{itemize}
where  $\rho_{V,W}: \Mm(V) \ra \Mm(W)$ are the restriction maps of the sheaf  $\Mm$,  $V,W \subset M$ open, and $*$ denotes the convolution product in $Hom(H,\Mm(U)).$ 
The maps $\tau_{ij}$, which in general are not algebra morphisms, are referred to as the transition functions of the quantum principal bundle. 

As it happens in classical geometry, also in the quantum case it is possible to reconstruct  a quantum principal bundle out of its transition functions. If it is given a Hopf algebra $H$, an open covering $\mathcal{U}=(U_i)_{i\in I}$  of $M$ and a sheaf $\Mm$ over $M$ with objects in $\ma$, then a family of linear maps $
\tau_{ij}:H \ra \Mm(U_i \cap U_j)$ fulfilling the conditions $(i)-(iii)$ above, also referred to as an $H$-cocycle, determines a quantum 
principal bundle  $\left( \P,\Mm,\sigma,H,(\Omega_i)_{i\in I} \right)$ over $M$. Assuming now that the linear maps $\tau_{ij}$ are algebra morphisms,   the quantum total space $\P$ is constructed like follows: for all $U$ open in $M$ the algebra $\P(U)$ is defined to be 
\begin{eqnarray}
\P(U)&:=& \left\{  (f_i) \in \oplus_{i \in I} \Mm(U_i \cap U) \ot H ~\big|~ \forall j,k \in I,  \quad 
 \left(\rho_{U_j \cap U , U_j \cap U_k \cap U} \ot id\right)(f_j) =
\right.
\\ &&
\left.
~=
(m \ot id) (id \ot \rho_{U_j \cap U_k, U_j \cap U_k \cap U} \circ  \tau_{j,k} \ot id )
 (id \ot \Delta) (\rho_{U_k \cap U, U_j \cap U_k \cap U} \ot id)(f_k) 
\right\}. \nn
\end{eqnarray}
The sheaf morphism $\sigma: \Mm \ra \P$ is then simply given by
$$\sigma_U: \Mm(U) \ra \P(U), \quad f \mapsto \sum_{i \in I} \rho_{U,U_i \cap U}(f) \ot 1
 $$ for each $U \subseteq M$ open. Finally the sheaf morphisms  $\Omega_i: \Mm_{|_{U_i}} \ot H \ra \P_{|_{U_i}}$, $i\in I$, are given by setting for each $U \subseteq U_i$
\begin{eqnarray}
&&(\Omega_i)_U: \Mm(U) \ot H \ra \P(U)~, \quad f \mapsto (f_k)_{k \in I} ,
\\
&& f_k:=  (m \ot id) (id \ot \rho_{U_i \cap U_k, U_i \cap U_k \cap U} \circ  \tau_{i,k}\ot id )
 (id \ot \Delta) (\rho_{U_i \cap U, U_i \cap U_k \cap U} \ot id)(f). \nn
\end{eqnarray}
\noindent
The maps $\Omega_i$ are bijective   and in particular they are 
isomorphisms of right $H$-comodules, for $\Mm(U) \ot H$ endowed with the coaction $id \ot \Delta$ and $\P$ with coaction \eqref{prin-coa}.

\section{Twisted tensor products}\label{se:app}
Consider two associative and unital algebras $(A,\cdot_A)$ and  $(B,\cdot_B)$,  $a,a' \in A$ and $b,b' \in B$. A twist map is just a linear map $\Psi:B\otimes A\ra A\otimes B$. Set $\Psi(b\otimes a)=a^{[\Psi]}\otimes b^{[\Psi]}$. To each twist map $\Psi$ one can associate a multiplication $$\cdot_{\Psi}=(\cdot_A\otimes\cdot_B)(\id_A\otimes\Psi\otimes\id_B)$$ in the tensor product $A\otimes B$. Algebras of the form $(A\otimes B,\cdot_{\Psi})$ are referred to as twisted tensor product algebras $A\otimes_{\Psi}B$. The ordinary tensor product algebra $(A\otimes B,\cdot_{\otimes})$ corresponds to  the flip map $\Psi=\tau$. A twist map is said to be normal if $\Psi(b\otimes 1_A)=1_A\otimes b$ and $\Psi(1_B\otimes a)=a\otimes 1_B$ for every $a,b$. There are necessary and/or sufficient conditions on $\Psi$ such that $\cdot_{\Psi}$ is unital and/or associative, see e.g. \cite{cae} for normal twists and \cite{cp13} for non-normal twists. In particular if $\Psi$ is normal and satisfies 
\begin{eqnarray}
&&(\id_A\ot m_B)(\Psi \ot \id_B)(\id_B \ot \Psi) =\Psi (m_B \ot \id_A )
\, , \label{P1-app}
\\
&& (m_A \ot \id_B)(\id_A\ot\Psi)(\Psi \ot \id_A) =\Psi (\id_B \ot m_A )\,  ,\label{P2-app}
\end{eqnarray}
then 
the multiplication $m_\theta$ is associative \cite{cae}.

We are interested in the compatibility of twisted tensor products with $*$-structures.
Let $A,B$ unital $*$-algebras with $*$-structures $*_A,~ *_B$ respectively. Let $*_\ot:=(*_A \ot *_B)$. 
\begin{defi} 
A map $\Psi': A \ot B \ra A \ot B$ is said to be $*$-compatible provided 
\be\label{*comp}
*_\ot = \Psi ' \circ *_\ot \circ \Psi '
\ee
\end{defi}
\begin{prop}\label{prop:*}
Let $\Psi: B \ot A \ra A \ot B $ be a {normal} twist. 
If $\Psi':=\Psi \circ \tau$ is $*$-compatible, then 
\be
*_\Psi :=  (\Psi \circ \tau) \circ *_\ot : a \ot b \mapsto \Psi(b^* \ot a^*) 
\ee
defines a $*$-structure in the twisted tensor algebra $A \ot_\Psi B$.
\end{prop}
\begin{proof}
It follows immediately from \eqref{*comp} that $*_\Psi \circ *_\Psi = id$. To prove that $*_\Psi$ is an algebra involution we proceed in different steps.  We denote by $\cdot_\Psi$ the multiplication in $A \ot_\Psi B$.  First, by using \eqref{*comp} and the fact that $\Psi$ is normal,  we have
$$
\left( (1 \ot b) \cdot_\Psi (a \ot 1) \right)^{*_\Psi}= \Psi' \circ *_\ot \left( \Psi (b \ot a) \right)= *_\ot (a \ot b) = 
 (a \ot 1)^{*_\Psi} \cdot_\Psi (1 \ot b)^{*_\Psi}  \, .
$$ 
Similarly, from the definition of $*_\Psi$,  we can prove the results on $ (a \ot 1) \cdot_\Psi (1 \ot b) $. 
Finally, before proving the result in full generality, we show it holds on $ (a \ot 1) \cdot_\Psi (a' \ot b) $ and $ (a \ot b) \cdot_\Psi (1 \ot b')$. For this we need to use the eqs. \eqref{P1-app} and \eqref{P2-app} above.  For instance 
\begin{eqnarray*}
\left(  (a \ot 1) \cdot_\Psi (a' \ot b) \right)^{*_\Psi} & =& \Psi (b^* \ot {a'}^* a^*) =
({a'}^*)^{[\Psi]} (a^*)^{[[\Psi]]} \ot ({b^*}^{[\Psi]})^{[[\Psi]]} \\
&=& \left( ({a'}^*)^{[\Psi]} \ot (b^*)^{[\Psi]} \right) \cdot_\Psi (a^* \ot 1) = (a' \ot b)^{*_\Psi}) \cdot_\Psi ( a \ot 1)^{*_\Psi}
\end{eqnarray*}
where in the second equality we have used \eqref{P2}.
Then, we can conclude that $*_\Psi$ is an involution on the generic element $ (a \ot b) \cdot_\Psi (a' \ot b')$ by using the hypothesis of normality of 
$\Psi$ to split the product as $ (a \ot b) \cdot_\Psi (a' \ot b')=  (a \ot 1) \cdot_\Psi (1 \ot b) \cdot_\Psi (a' \ot 1) \cdot_\Psi (1 \ot b')$ and thus by applying the above intermediate results.
\end{proof}

\end{document}